\documentclass[a4paper, reqno, 11pt]{amsart} %a4paper

\usepackage[utf8]{inputenc}
\usepackage[T1]{fontenc}

\usepackage[english]{babel}

\usepackage[margin=1.2in]{geometry}
\usepackage{enumitem}
\usepackage[bookmarks = true]{hyperref}
\usepackage[all]{xy}
\usepackage{xcolor}
\usepackage{amsmath, amsthm, amssymb, bbm, varwidth}
\usepackage{mathtools}
\usepackage{unicode-math}

\newtheorem{thm}{Theorem}[section]

\newtheorem*{thm*}{Theorem}
\newtheorem*{thmA*}{Theorem A}
\newtheorem*{thmB*}{Theorem B}
\newtheorem{lem}[thm]{Lemma}
\newtheorem{prop}[thm]{Proposition}
\newtheorem{cor}[thm]{Corollary}
\newtheorem*{cor*}{Corollary}

\theoremstyle{definition}
\newtheorem{rmk}[thm]{Remark}

\newtheorem{defn}[thm]{Definition}
\newtheorem{ex}[thm]{Example}
\newtheorem{conj}[thm]{Conjecture}
\newtheorem{construction}[thm]{Construction}

\numberwithin{equation}{section}
\newcommand{\lemm}[1]{\begin{lem}#1\end{lem}}

\newcommand{\defi}[1]{\begin{defn}#1\end{defn}}
\newcommand{\pf}[1]{\begin{proof}#1\end{proof}}
\newcommand{\eqn}[1]{\begin{equation}#1\end{equation}}

\newcommand{\overbar}[1]{\mkern 1.5mu\overline{\mkern-1.5mu#1\mkern-1.5mu}\mkern 1.5mu}

\newcommand{\lr}{\longrightarrow}

\newcommand{\mr}{\mathrm}

\newcommand{\wt}{\widetilde}

\newcommand{\et}{\mathrm{et}}

\newcommand{\Orb}{\operatorname{Orb}}

\newcommand{\diag}{\operatorname{diag}}

\newcommand{\Vol}{\operatorname{Vol}}

\newcommand{\End}{\operatorname{End}}

\newcommand{\Lie}{\operatorname{Lie}}

\newcommand{\Spec}{\operatorname{Spec}}

\newcommand{\Spf}{\operatorname{Spf}}

\newcommand{\Gal}{\operatorname{Gal}}

\newcommand{\tensor}{\otimes}

\newcommand{\iso}{\cong}

\newcommand{\mbC}{\mathbb{C}}

\newcommand{\mbF}{\mathbb{F}}

\newcommand{\mbN}{\mathbb{N}}

\newcommand{\mbQ}{\mathbb{Q}}
\newcommand{\mbR}{\mathbb{R}}

\newcommand{\mbX}{\mathbb{X}}
\newcommand{\mbY}{\mathbb{Y}}
\newcommand{\mbZ}{\mathbb{Z}}

\newcommand{\mcI}{\mathcal{I}}

\newcommand{\mcL}{\mathcal{L}}
\newcommand{\mcM}{\mathcal{M}}
\newcommand{\mcN}{\mathcal{N}}
\newcommand{\mcO}{\mathcal{O}}

\newcommand{\mcZ}{\mathcal{Z}}

\newcommand{\sL}{\mathcal{L}}

\makeatletter
\let\save@mathaccent\mathaccent
\newcommand*\if@single[3]{%
  \setbox0\hbox{${\mathaccent"0362{#1}}^H$}%
  \setbox2\hbox{${\mathaccent"0362{\kern0pt#1}}^H$}%
  \ifdim\ht0=\ht2 #3\else #2\fi
  }
\newcommand*\rel@kern[1]{\kern#1\dimexpr\macc@kerna}
\newcommand*\widebar[1]{\@ifnextchar^{{\wide@bar{#1}{0}}}{\wide@bar{#1}{1}}}
\newcommand*\wide@bar[2]{\if@single{#1}{\wide@bar@{#1}{#2}{1}}{\wide@bar@{#1}{#2}{2}}}
\newcommand*\wide@bar@[3]{%
  \begingroup
  \def\mathaccent##1##2{%
    \let\mathaccent\save@mathaccent
    \if#32 \let\macc@nucleus\first@char \fi
    \setbox\z@\hbox{$\macc@style{\macc@nucleus}_{}$}%
    \setbox\tw@\hbox{$\macc@style{\macc@nucleus}{}_{}$}%
    \dimen@\wd\tw@
    \advance\dimen@-\wd\z@
    \divide\dimen@ 3
    \@tempdima\wd\tw@
    \advance\@tempdima-\scriptspace
    \divide\@tempdima 10
    \advance\dimen@-\@tempdima
    \ifdim\dimen@>\z@ \dimen@0pt\fi
    \rel@kern{0.6}\kern-\dimen@
    \if#31
      \overline{\rel@kern{-0.6}\kern\dimen@\macc@nucleus\rel@kern{0.4}\kern\dimen@}%
      \advance\dimen@0.4\dimexpr\macc@kerna
      \let\final@kern#2%
      \ifdim\dimen@<\z@ \let\final@kern1\fi
      \if\final@kern1 \kern-\dimen@\fi
    \else
      \overline{\rel@kern{-0.6}\kern\dimen@#1}%
    \fi
  }%
  \macc@depth\@ne
  \let\math@bgroup\@empty \let\math@egroup\macc@set@skewchar
  \mathsurround\z@ \frozen@everymath{\mathgroup\macc@group\relax}%
  \macc@set@skewchar\relax
  \let\mathaccentV\macc@nested@a
  \if#31
    \macc@nested@a\relax111{#1}%
  \else
    \def\gobble@till@marker##1\endmarker{}%
    \futurelet\first@char\gobble@till@marker#1\endmarker
    \ifcat\noexpand\first@char A\else
      \def\first@char{}%
    \fi
    \macc@nested@a\relax111{\first@char}%
  \fi
  \endgroup
}
\newcommand{\ord}{\mathrm{ord}}

\renewcommand{\O}{O}

\DeclarePairedDelimiter\abs{\lvert}{\rvert}%
\DeclarePairedDelimiter\norm{\lVert}{\rVert}%

\makeatletter
\let\oldabs\abs
\def\abs{\@ifstar{\oldabs}{\oldabs*}}
\let\oldnorm\norm
\def\norm{\@ifstar{\oldnorm}{\oldnorm*}}
\makeatother

\newcommand{\del}{\operatorname{\partial Orb}}

\include{header}

\begin{document}

\title{A linear AFL for quaternion algebras}
\author{Nuno Hultberg and Andreas Mihatsch}
\date{\today}

\address{Nuno Hultberg, University of Copenhagen, Department of Mathematical Sciences, Universitetsparken 5, 2100 København, Denmark}
\email{nh@math.ku.dk}
\address{Andreas Mihatsch, Universität Bonn, Mathematisches Institut, Endenicher Allee 60, 53115 Bonn, Germany}
\email{mihatsch@math.uni-bonn.de}

\begin{abstract}
We prove new fundamental lemma and arithmetic fundamental lemma identities for general linear groups over quaternion division algebras. In particular, we verify the transfer conjecture and the arithmetic transfer conjecture from \cite{LM2} in cases of Hasse invariant $1/2$.
\end{abstract}

\maketitle

\tableofcontents

\newcommand{\Inv}{\mr{Inv}}
\newcommand{\quand}{\quad \text{and}\quad}
\newcommand{\Iw}{\mr{Iw}}
\newcommand{\Par}{\mr{Par}}
\newcommand{\bdot}{\,\cdot\,}
\newcommand{\rs}{\mr{rs}}

\section{Introduction}

Fix a non-archimedean local field $F$ and some $n\geq 1$. The linear arithmetic fundamental lemma (AFL) conjecture of Q. Li \cite{Li} states a family of identities between derivatives of orbital integrals on $GL_{2n}(F)$ and intersection numbers on moduli spaces of strict formal $O_F$-modules of height $2n$. It has a global motivation which is parallel to that of W. Zhang's unitary AFL \cite{Z} and which is related to the trace formula comparison of Leslie--Xiao--Zhang \cite{LXZ}. We refer to \cite{Z_survey_18} and the introduction of \cite{LM2} for this global aspect and henceforth focus on the local setting.

While the unitary AFL has been proved \cite{Z19, M_loc_const, MZ, ZZ}, the linear AFL conjecture is still open. It is, however, known to hold when $n\leq 2$, see \cite{Li_GL4}. Moreover, both the conjecture and its validity for $n\leq 2$ have been extended to a biquadratic setting by Howard--Li \cite{HL} and Li \cite{Li_future}. A non-basic version of both the linear and the biquadratic AFL is formulated and reduced to the basic setting in \cite{LM1}.

The linear AFL concerns orbital integrals for hyperspecial test functions and moduli spaces for $GL_{2n}$ for hyperspecial level (good reduction). In a recent article, Li and the second author formulated a variant that relates parahoric test functions on the analytic side with moduli spaces for central simple algebras on the intersection-theoretic side.

In the present article, we consider this variant in the case of Hasse invariant $1/2$. More precisely, we prove a fundamental lemma (FL) type and an AFL type statement for the group $GL_n(B)$ where $B/F$ is a quaternion division algebra. We now state these results in a vague form, together with references to their precise formulations.

\begin{thmA*}[Fundamental Lemma, see Theorem \protect{\ref{thm:FL}}]
Let $γ\in GL_{2n}(F)$ and $g\in GL_n(B)$ be regular semi-simple matching elements. Then
$$\Orb(γ, 1_\Par, 0) = \Orb(g, 1_{GL_n(O_B)}).$$
\end{thmA*}
Here, the orbital integral on the left is with respect to the subgroup $GL_n(F\times F)$. The one on the right is with respect to the subgroup $GL_n(E)$ for a fixed embedding $E\to B$ of an unramified quadratic extension $E/F$. The test function $1_\Par$ on the left hand side is the indicator function of the standard $(n\times n)$-block parahoric subgroup of $GL_{2n}$. Precise definitions will be given in \S\ref{s:results}.

In order to state our arithmetic result, we also need the datum of a $1$-dimensional strict $O_E$-module $\mbY$ of $O_E$-height $n$. The Serre tensor construction $O_B\tensor_{O_E} \mbY$ then has dimension $2$ and $O_F$-height $4n$.

\begin{thmB*}[Arithmetic Fundamental Lemma, see Theorem \protect{\ref{thm:main}}]
Let $γ\in GL_{2n}(F)$ and $g\in \mr{Aut}^0_B(O_B\tensor_{O_E} \mbY)$ be regular semi-simple matching elements. Assume that the linear AFL holds for $\mbY$. Then
\begin{equation}\label{eq:afl_intro}
\left.\frac{d}{ds}\right\vert_{s = 0} \Orb(γ, 1_\Par, s) = 2\, \mr{Int}(g) \log(q).
\end{equation}
\end{thmB*}

We refer to \S\ref{ss:def_intersection} for the definition of the intersection number $\mr{Int}(g)$. Using that the linear AFL is known for all $\mbY$ whose connected factor has $O_E$-height $\leq 2$ we obtain the following corollary. 

\begin{cor*}[see Corollary \ref{cor:AFL_low_dim}]
Assume that the connected factor of $\mbY$ has $O_E$-height $\leq 2$. Then, for every pair of regular semi-simple matching elements $γ\in GL_{2n}(F)$ and $g\in \mr{Aut}^0_B(O_B\tensor_{O_E} \mbY)$,
$$\left.\frac{d}{ds}\right\vert_{s = 0} \Orb(γ, 1_\Par, s) = 2\, \mr{Int}(g) \log(q).$$
\end{cor*}

In particular, this verifies new cases of the arithmetic transfer (AT) conjecture \cite[Conjecture 1.5]{LM2}. We remark that \cite[Conjecture 1.5]{LM2} involves an unspecified correction term. Identity \eqref{eq:afl_intro} shows that we expect this term to vanish for Hasse invariant $1/2$.

Our proofs of Theorems A and B are by reduction to the Guo--Jacquet FL \cite{Guo}, the quadratic base change FL, and the linear AFL. On the orbital integral side, this reduction relies on a combinatorial interpretation in terms of lattice counts. On the intersection-theoretic side, it relies on a moduli-theoretic translation between the intersection problems for $GL_{2n}(F)$ and $GL_n(B)$.

\section{Statement of Results}
\label{s:results}

The definitions and conventions that follow are taken from \cite[\S2 -- \S4]{LM2}, but specialized to general linear groups over quaternion algebras.
We begin this section by recalling the invariant polynomial (\S\ref{ss:groups}) and by fixing our setting (\S\ref{ss:setting}). Next, we give the definitions of the relevant orbital integrals (\S\ref{ss:orb_int}) and intersection numbers (\S\ref{ss:def_intersection}). Then we state our main results (\S\ref{ss:results}).

\subsection{Regular semi-simple orbits}
\label{ss:groups}
We will consider several instances of the following kind. Let $F$ be a field and let $E/F$ be an étale quadratic extension with Galois conjugation $σ$. Let $D$ be a central simple algebra (CSA) over $F$ of degree $2n$ together with an embedding $E\to D$. Let $C = \mr{Cent}_E(D)$ be the centralizer which is a CSA over $E$ of degree $n$.

In this situation, $C^\times \times C^\times$ acts on $D^\times$ by $(h_1,h_2)\cdot g = h_1^{-1}gh_2$. An element $g\in D^\times$ is called regular semi-simple if its orbit for this action is Zariski closed and its stabilizer of minimal dimension. (This minimal dimension is equal to $n$.) According to Jacquet--Rallis \cite{JR} and Guo \cite[\S1]{Guo}, the regular semi-simple orbits can be characterized as follows:

Let $D = D_+\oplus D_-$ be the decomposition into $E$-linear and $E$-conjugate linear components. That is, $D_+ = C$ and $D_- = \{x\in D \mid xa = σ(a)x\text{ for }a\in E\}$. Let $g = g_++g_-$ denote the corresponding decomposition of an element $g\in D^\times$. For $g\in D^\times$ with $g_+\in D^\times$, set $z_g = g_+^{-1}g_-$.
\begin{defn}\label{def:invariant}
Assume that $g\in D^\times$ with $g_+\in D^\times$. The reduced characteristic polynomial $\mr{charred}_{D/F}(z_g^2; T) \in F[T]$ is always a square. Define the invariant of $g$ as its unique monic square root,
\begin{equation}
\Inv(g; T) := \mr{charred}_{D/F}(z_g^2;T)^{1/2} \in F[T].
\end{equation}
\end{defn}
The polynomial $\Inv(g;T)$ is monic, of degree $n$, and satisfies $\Inv(g;1) \neq 0$. An element $g\in D^\times$ is regular semi-simple if and only if both $g_+, g_-$ lie in $D^\times$ and $\Inv(g;T)$ is a separable polynomial. The invariant polynomial classifies orbits in the sense that for two regular semi-simple elements $g_1, g_2\in D^\times$,
$$C^\times g_1 C^\times = C^\times g_2 C^\times\quad \Longleftrightarrow\quad \Inv(g_1; T) = \Inv(g_2;T).$$

\subsection{Setting and notation}
\label{ss:setting}
For the rest of this article, we fix an integer $n\geq 1$ and a non-archimedean local field $F$ with uniformizer $\pi$. Let $q$ denote its residue field cardinality, $v$ its normalized valuation and $|x| = q^{-v(x)}$ its normalized absolute value.

We denote by $K = F\times F$ the split quadratic extension of $F$. The diagonal embedding $K\to M_{2n}(F),\ (a,b)\mapsto \mr{diag}(a\cdot 1_n, b\cdot 1_n)$ is of the kind considered in \S\ref{ss:groups} and has centralizer $M_n(K)$. Let
$$(G',\ H') = (GL_{2n}(F),\ GL_n(K))$$
be the corresponding pair of linear groups. For $γ\in G'$, the notion of being regular semi-simple and the invariant $\Inv(γ;T)$ are meant with respect to the $(H'\times H')$-action. Let $G'_{\rs}\subset G'$ denote the set of regular semi-simple elements.

We denote by $E/F$ an unramified quadratic extension, by $σ\in \mr{Gal}(E/F)$ the non-trivial element, and by $η:F^\times \to \{\pm 1\}$, $η(x) = (-1)^{v(x)}$ the character associated to $E$. Let $O_F$ and $O_E$ be the rings of integers in $F$ and $E$.

Let $B_λ/F$ be a quaternion algebra over $F$ of Hasse invariant $λ\in \{0,\, 1/2\}$. For both possibilities $B\in \{B_0, B_{1/2}\}$, we fix a maximal order $O_B\subset B$ and an embedding $O_E \to O_B$. The resulting diagonal embedding $E \to M_n(B)$ is of the type considered in \S\ref{ss:groups} and has centralizer $M_n(E)$. Let
$$(G_λ,\ H) = (GL_n(B_λ),\ GL_n(E))$$
be the corresponding pair of linear groups. For $g\in G_λ$, the notion of being regular semi-simple and the invariant $\Inv(g; T)$ are meant with respect to the $(H\times H)$-action. Let $G_{λ,\rs}\subset G_λ$ be the subset of regular semi-simple elements.

\subsection{Orbital integrals}
\label{ss:orb_int}

\subsubsection{The case of $G'$}
We define two characters $η$ and $|\bdot|$ on $H'$ by
$$η,\ |\bdot|:H'\lr \mbC^\times, \quad η((a,b)) := η(\det(ab^{-1})),\quad |(a,b)| := |\det(ab^{-1})|.$$
For a regular semi-simple element $γ\in G'_{\rs}$, we denote by 
\begin{equation}\label{eq:stabilizer}
(H'\times H')_γ = \{(h_1, h_2) \in H'\times H' \mid h_1γ = γh_2\}
\end{equation}
the stabilizer of $γ$. Set $L_γ = F[z_γ^2]$ which is an étale extension of degree $n$ of $F$ because $γ$ is regular semi-simple. There is the identity $L_γ^\times = H'\cap γ^{-1}H'γ$, so $(H'\times H')_γ$ can be identified with the torus $L_γ^\times$. We normalize the Haar measures on $H'\times H'$ and $(H'\times H')_γ$ by
\begin{equation}\label{eq:Haar_1}
\mr{Vol}(GL_n(O_K)) = \mr{Vol}(O_{L_γ}^\times) = 1.
\end{equation}
For a regular semi-simple element $γ\in G'_{\rs}$, a test function $f'\in C^\infty_c(G')$ and a complex parameter $s\in \mbC$, we can now define the orbital integral
\begin{equation}\label{eq:def_orb_int_1}
\Orb(γ,f',s) := \Omega(γ, s)\cdot  \int_{\frac{H'\times H'}{(H'\times H')_γ}} f'(h_1^{-1} γ h_2) |h_1h_2|^s η(h_2) d(h_1,h_2).
\end{equation}
Here, the so-called transfer factor $\Omega(γ, s)\in \pm q^{\mbZ\cdot s}$ ensures that the definition only depends on the orbit $H'γH'$. It is defined by
$$\Omega\left(\begin{pmatrix}
a & b \\ c & d
\end{pmatrix}, s\right) = η(\det(cd^{-1}))\cdot |\det(b^{-1}c)|^s$$
and satisfies $\Omega(h_1^{-1}γh_2, s) = |h_1h_2|^sη(h_2)\Omega(γ, s).$ The definition in \eqref{eq:def_orb_int_1} moreover relies on \cite[Lemma 3.2.3.]{HL} which states that $η$ and $|\bdot|$ have trivial restriction to $(H'\times H')_γ$. We will be interested in the central value and the central derivative only, so we define
\begin{equation}\label{eq:orb_int_central}
\Orb(γ, f') := \Orb(γ, f', 0)\quand \del (γ, f') := \left.\frac{d}{ds}\right\vert_{s = 0}\Orb(γ, f', s).
\end{equation}

\subsubsection{The case of $G_λ$}
Let $G \in \{ G_0, G_{1/2}\}$ be one of the two possible groups. For a regular semi-simple element $g\in G_{\rs}$, we denote by $(H\times H)_g$ its stabilizer. It can be identified with the torus $L_g^\times$, where $L_g = F[z_g^2]$ as before. We normalize the Haar measures on $H\times H$ and $(H\times H)_g$ by
\begin{equation}\label{eq:Haar_2}
\mr{Vol}(GL_n(O_E)) = \mr{Vol}(O_{L_g}^\times) = 1.
\end{equation}
For a regular semi-simple element $g\in G_{\rs}$ and a test function $f\in C^\infty_c(G)$, we put
\begin{equation}\label{eq:def_orb_int_2}
\Orb(g,f) := \int_{\frac{H\times H}{(H\times H)_g}} f(h_1^{-1}gh_2) d(h_1,h_2).
\end{equation}

\subsection{Intersection numbers}
\label{ss:def_intersection}
We next define two families of intersection numbers, one for each of the possible Hasse invariants $λ \in \{0, 1/2\}$. For $λ = 0$, these are the intersection numbers that occur in the linear AFL from \cite{LM1}. For $λ = 1/2$, these are a special case of the intersection numbers from \cite{LM2}.

Let $B = B_λ$ in the following. Let $\breve F$ be the completion of a maximal unramified extension of $F$ with ring of integers $O_{\breve F}$ and residue field $\mbF$; fix an embedding $E\to \breve F$. As a further datum, let $\mbY$ be a strict $π$-divisible $O_E$-module of dimension $1$ and $O_E$-height $n$ over $\mbF$, see \cite[Definition 3.1]{LM1}. We denote its $O_E$-action by $\iota:O_E\to\End(\mbY)$. Recall that the connected-étale sequence of any such $\mbY$ has a unique splitting, meaning there is a canonical isomorphism
$$\mbY = \mbY^\circ \times \mbY^\et$$
where $\mbY^\circ$ is the connected component of the identity of $\mbY$ and $\mbY^\et$ its maximal étale quotient. Let $n^\circ$ and $n^\et$ be the $O_E$-heights of $\mbY^\circ$ and $\mbY^\et$. Then $n = n^\circ + n^\et$, every $1\leq n^\circ \leq n$ may occur, and $(n^\circ, n^\et)$ characterize $\mbY$ uniquely up to isomorphism.

Define $(\mbX, κ)$ as the Serre tensor product $\mbX = O_B \tensor_{O_E}\mbY$ with $O_B$-action $κ(x) = x\tensor \mr{id}_\mbY$. Note that $\mbX$ also carries the $O_B$-linear $O_E$-action $\mr{id}_{O_B}\tensor ι$ which we again denote by $ι$.

\begin{rmk}\label{rmk:Morita}
Assume that $λ = 0$. Then any choice of isomorphism $O_B \iso M_2(O_F)$ provides a decomposition of $\mbX$ as $\mbX = \mbY^{\oplus 2}$ (Morita equivalence). In this way, the ensuing definitions for $λ = 0$ reduce to the ones in \cite{LM1}.
\end{rmk}

We consider the rings of quasi-endomorphisms $C_\mbY = \End^0_E(\mbY, ι)$ and $D_{λ,\mbY} = \End^0_B(\mbX, κ)$, as well as the groups $H_\mbY = C_\mbY^\times$ and $G_{λ,\mbY} = D_{λ,\mbY}^\times$. By functoriality of the Serre tensor construction, there is a group homomorphism
$$H_\mbY\lr G_{λ,\mbY},\quad g\longmapsto \mr{id}_{O_B}\tensor g.$$
The structure of these groups is as follows. Let $\mbX = \mbX^\circ\times \mbX^\et$ be the connected-étale decomposition of $\mbX$. There are no homomorphisms between connected and étale $π$-divisible groups over $\mbF$, so
\begin{equation}\label{eq:decomp_C_D}
C_\mbY = C_\mbY^\circ \times C_\mbY^\et,\ \ \quad D_{λ,\mbY} = D^\circ_{λ,\mbY}\times D_{λ,\mbY}^\et,
\end{equation}
where the factors denote the $E$-linear (resp. $B$-linear) quasi-endomorphisms of the factors of $\mbY$ and $\mbX$. Then
$$\begin{array}{rclrcl}
C_\mbY^\circ & \iso & C_{1/n^\circ} \quad & D_{λ,\mbY}^\circ & \iso & D_{1/2n^\circ + λ}\\[4pt]
C_\mbY^\et & \iso & M_{n^\et}(E) \quad & D_{λ,\mbY}^\et & \iso & M_{n^\et}(B).
\end{array}$$
Here, $C_{1/n^\circ}$ denotes a central division algebra (CDA) of Hasse invariant $1/n^\circ$ over $E$, and $D_{1/2n^\circ + λ}$ denotes a CSA of degree $2n^\circ$ and Hasse invariant $1/2n^\circ + λ$ over $F$. Note that the two pairs $(C_\mbY^\circ, D_{λ,\mbY}^\circ)$ and $(C_\mbY^\et, D_{λ,\mbY}^\et)$ are of the type considered in \S\ref{ss:groups}.
\begin{defn}\label{eq:reg_ss_II}
An element $g = (g^\circ, g^\et)\in G_{λ,\mbY}$ is called regular semi-simple if its components $g^\circ$ and $g^{\mr{et}}$ are regular semi-simple with respect to $H_\mbY^\circ$ and $H_\mbY^\et$, and if its invariant, defined as
$$\Inv(g; T) = \Inv(g^\circ; T)\,\Inv(g^{\mr{et}}; T) \in F[T],$$
is a separable polynomial. Let $G_{λ, \mbY, \mr{rs}}\subset G_{λ, \mbY}$ be the subset of regular semi-simple elements.
\end{defn}

We now associate an intersection number to each $g\in G_{λ,\mbY, \mr{rs}}$. Let $\mcN$ be the RZ space for $\mbY$: By definition, this means that it is the formal scheme over $O_{\breve F}$ that represents the functor
\begin{equation}\label{eq:def_LT_space}
\mcN(S) = \left\{(Y, ι, ρ)\ \left\vert\ \text{\begin{varwidth}{\textwidth}\centering 
$(Y,ι)/S$ a strict $O_E$-module\\
$ρ:\overbar{S}\times_{\Spec \mbF} (\mbY, ι)\lr \overbar{S}\times_S (Y,ι)$ a quasi-isogeny
\end{varwidth}}\right.\right\}.
\end{equation}
Here and in the following, $\overbar{S} = \mbF\tensor_{O_{\breve F}} S$ denotes the special fiber of $S$.

Similarly, let $\mcM_λ$ be the RZ space for $(\mbX, κ)$. Recall for its definition that an $O_B$-action on a $2$-dimensional strict $O_F$-module $X$ over a $\Spf O_{\breve F}$-scheme $S$ is called special if the two $κ(O_E)$-eigenspaces of $\Lie(X)$ are both locally free of rank $1$ as $\mcO_S$-modules.\footnote{This definition goes back to Drinfeld \cite{Drinfeld}.} Then $\mcM_λ$ is the formal scheme over $\Spf O_{\breve F}$ that represents the functor
\begin{equation}\label{eq:def_mysterious_space}
\mcM_λ(S) = \left\{(X, κ, ρ)\ \left\vert\ 
\text{\begin{varwidth}{\textwidth}\centering 
$(X,κ)/S$ a special $O_B$-module\\
$ρ:\overbar{S} \times_{\Spec \mbF} (\mbX, κ) \lr \overbar{S}\times_S (X,κ)$ a quasi-isogeny
\end{varwidth}}\right. \right\}.
\end{equation}
It is clear that if $(Y, ι)$ is a $1$-dimensional strict $O_E$-module, then $O_B\tensor_{O_E}(Y, ι)$ is a special $O_B$-module. Thus we obtain a morphism
\begin{equation}\label{eq:Serre}
\begin{aligned}
\mcN & \longrightarrow \ \mcM\\
(Y,\,ι,\,ρ) & \longmapsto \ (O_B\tensor_{O_E} Y,\, κ(x) := x\tensor \mr{id}_Y,\, \mr{id}_{O_B}\tensor ρ).
\end{aligned}
\end{equation}
This is a closed immersion by \cite[Proposition 4.15]{LM2}. There are right actions of $H_\mbY$ on $\mcN$ and of $G_{λ,\mbY}$ on $\mcM_λ$ by
$$h\cdot (Y, ι, ρ) = (Y, ι, ρh)\quand g\cdot (X, κ, ρ) = (X, κ, ρg).$$
The morphism \eqref{eq:Serre} is equivariant with respect to $H_\mbY \to G_{λ,\mbY}$. Moreover, the structure of $\mcN$ as formal scheme is easy to describe:
\begin{equation}\label{eq:str_N}
\mcN\iso \coprod_{\mbZ\, \times\, GL_{n^{\mr{et}}}(E)/GL_{n^{\mr{et}}}(O_E)} \Spf O_{\breve F}[\![t_1,\ldots,t_{n-1}]\!],
\end{equation}
where the indexing can be chosen compatibly with the $H_\mbY^\et$-action for a fixed identification $H_\mbY^\et \iso GL_{n^\et}(E)$. In particular, $\mcN$ is formally smooth of relative dimension $n-1$ over $\Spf O_{\breve F}$.

Concerning $\mcM_λ$, it is known to be locally formally of finite type over $\Spf O_{\breve F}$ and regular of dimension $2n$, see \cite[Proposition 4.13]{LM2}. Using \eqref{eq:Serre}, we can thus view $\mcN$ as cycle in middle dimension on $\mcM_λ$.

For a regular semi-simple element $g\in G_{λ,\mbY,\rs}$, we define the intersection locus
\begin{equation}\label{eq:def_Ig}
\mcI(g) := \mcN\cap g\cdot \mcN.
\end{equation}
For such $g$, we also denote by $g = g_+ + g_-$ the decomposition of $g$ into $E$-linear and $E$-conjugate linear components, where $ι:E\to D_{λ,\mbY}$ comes from the definition of $(\mbX, κ)$. Set $z_g = g_+^{-1}g_-$ and $L_g = F[z_g^2]$. Since $L_g^\times = H_\mbY \cap g^{-1}H_\mbY g$ as subgroups of $G_{λ,\mbY}$, the $L_g^\times$-action on $\mcM$ preserves the intersection locus $\mcI(g)$. Let $Γ \subset L_g^\times$ be a free discrete subgroup such that $L_g^\times = Γ\times O_{L_g}^\times$. It acts without fixed points on $\mcM_λ$.
\begin{prop}\label{prop:qc}
Let $g\in G_{λ,\mbY}$ and $Γ\subset L_g^\times$ be as before. The quotient $Γ\backslash (\mcN \cap g \cdot \mcN)$ is an artinian scheme.
\end{prop}
\begin{proof}
By \cite[Proposition 4.18]{LM2},\footnote{Strictly speaking, \cite[Proposition 4.18]{LM2} is formulated only for $p$-adic $F$. However, Proposition \ref{prop:qc} is known when $λ = 0$ by \cite[Lemma 3.7]{LM1} and the comparison between the cases $λ = 0$ and $λ = 1/2$ in \S\ref{s:AFL} gives an alternative proof of Proposition \ref{prop:qc} that applies to all $F$.} the quotient $Γ\backslash \mcI(g)$ is a proper scheme over $\Spec O_{\breve F}$ with empty generic fiber. Since the maximal reduced subscheme $\mcN_{\mr{red}}$ is $0$-dimensional, cf. \eqref{eq:str_N}, this scheme has to be artinian.
\end{proof}

The general definition of intersection numbers in \cite[Definition 4.21]{LM2} now specializes to taking the length:

\begin{defn}\label{def:intersection}
For a regular semi-simple element $g\in G_{λ,\mbY, \rs}$, choose $Γ$ as before and define
$$\mr{Int}(g) := \mr{len}_{O_{\breve F}} \left(\mcO_{Γ\backslash \mcI(g)} \right).$$
\end{defn}

\subsection{FL and AFL}
\label{ss:results}
We consider the following test functions. For the case $λ = 0$, i.e. for $B\iso M_2(F)$, we define
$$f'_0 = 1_{GL_{2n}(O_F)} \in C^\infty_c(G')\quand f_0 = 1_{GL_{2n}(O_F)} \in C^\infty_c(G).$$
For the case $λ = 1/2$, we first define the parahoric subgroup
\begin{equation}\label{eq:def_Par}
\Par = \left\{\begin{pmatrix}
GL_n(O_F) & π\,M_n(O_F) \\ M_n(O_F) & GL_n(O_F)\end{pmatrix}\right\} \subset GL_{2n}(O_F).
\end{equation}
Let $1_\Par \in C^\infty_c(G')$ be its characteristic function and let $h\in H'$ be the element $\diag(π 1_n, 1_n)$. Then put
\begin{equation}\label{eq:test_functions_1_2}
f'_{1/2}(\bdot) = 1_\Par(h\bdot) \in C^\infty_c(G') \quand f_{1/2} = 1_{GL_n(O_B)} \in C^\infty_c(G).
\end{equation}
The orbital integrals of $f'_{1/2}$ and $1_\Par$ are related by
\begin{equation}\label{eq:orb_int_normalization}
\Orb(γ, f'_{1/2}, s) = q^{-ns}\Orb(γ, 1_\Par, s).
\end{equation}
The advantage of $f'_{1/2}$ over $1_\Par$ is that its orbital integral satisfies the completely symmetric functional equation (see \cite[Proposition 3.19]{LM2})
\begin{equation}\label{eq:fctl_equation}
\Orb(γ, f'_{1/2}, -s) = ε_{1/2}(γ)\Orb(γ, f'_{1/2}, s)
\end{equation}
where the sign is defined by
\begin{equation}\label{eq:sign}
ε_{1/2}(γ) = (-1)^n\, (-1)^r,\quad r = v\left(\mr{det}_{M_n(B)/F}(z_γ)\right).
\end{equation}

\begin{defn}\label{def:matching}
Two regular semi-simple elements $γ\in G'_\rs$ and $g\in G_{λ,\rs}$ (resp. $γ$ and $g\in G_{λ,\mbY,\rs}$) are said to match if
$$\Inv(γ;T) = \Inv(g;T).$$
\end{defn}

\begin{thm}[Fundamental Lemma]\label{thm:FL}
For every regular semi-simple $γ\in G'_\rs$,
$$\Orb(γ, f'_λ) = \begin{cases} \Orb(g, f_λ) & \text{if there exists a matching $g\in G_{λ, \rs}$}\\
0 & \text{otherwise.}
\end{cases}$$
\end{thm}
The case $λ = 0$ is well-known (Guo--Jacquet FL) and due to Guo \cite{Guo}. Our addition is the case $λ = 1/2$ whose proof will be given in \S\ref{s:FL}. We now turn to the central derivatives:

\begin{conj}[Arithmetic Fundamental Lemma]\label{conj:AFL}
For every regular semi-simple element $γ\in G'_\rs$,
\begin{equation}\label{eq:AFL}
\del(γ, f'_λ) = \begin{cases} 2\,\mr{Int}(g)\log(q) & \text{\begin{varwidth}{\textwidth}if there is a strict $O_E$-module $(\mbY, ι)$\\
and $g\in G_{λ,\mbY, \rs}$ that matches $γ$\end{varwidth}}\\[8pt]
0 & \text{otherwise.}
\end{cases}
\end{equation}
\end{conj}
The AT conjecture for general CSAs \cite[Conjecture 1.5]{LM2} also includes an unspecified correction term that cannot be omitted in general. Conjecture \ref{conj:AFL} here, which is for Hasse invariant $1/2$, is hence stronger in the sense that this correction term is conjectured to vanish.

For $λ = 0$, Conjecture \ref{conj:AFL} is the linear AFL conjecture from \cite{LM1}. The vanishing part of \eqref{eq:AFL} is known to hold in general, see \cite[Corollary 2.18]{LM1}. Moreover, Identity \eqref{eq:AFL} is known for all $g\in G_{0, \mbY, \mr{rs}}$ with $\mbY$ such that $n^\circ \leq 2$. This statement is \cite[Corollary 1.3]{LM1} and goes back to Li \cite{Li_GL4} who verified the case $n \leq 2$.

Our main result here is a reduction of the case $λ = 1/2$ in Conjecture \ref{conj:AFL} to the case $λ = 0$:

\begin{thm}\label{thm:main}
\begin{enumerate}[wide, labelindent=0pt, labelwidth=!, label=(\arabic*), topsep=2pt, itemsep=2pt]
\item The vanishing part of Conjecture \ref{conj:AFL} holds.
\item If Conjecture \ref{conj:AFL} holds for $λ = 0$, then it also holds for $λ = 1/2$. More precisely, assume that Identity \eqref{eq:AFL} holds for $λ = 0$ and all $g\in G_{0, \mbY, \mr{rs}}$ for some $\mbY$. Then Identity \eqref{eq:AFL} also holds for $λ = 1/2$ and all $g\in G_{1/2, \mbY, \mr{rs}}$.
\end{enumerate}
\end{thm}

\begin{cor}[to Theorem \ref{thm:main} and \protect{\cite[Corollary 1.3]{LM1}}]\label{cor:AFL_low_dim}
Identity \eqref{eq:AFL} holds in all cases with $λ = 1/2$ and with $\mbY$ such that $n^\circ \leq 2$.\qed
\end{cor}

Our proofs of Theorems \ref{thm:FL} and \ref{thm:main} are by expressing the occurring orbital integrals and intersection numbers for $λ = 1/2$ in terms of orbital integrals and intersection numbers for $λ = 0$. This is made precise by the following result.

\begin{thm}\label{thm:auxiliary}
\begin{enumerate}[wide, labelindent=0pt, labelwidth=!, label=(\arabic*), topsep=2pt, itemsep=2pt]
\item Assume that $γ\in G'_{\mr{rs}}$ is regular semi-simple and that $π^{-n}\Inv(γ;πT)$ lies in $T^n + πO_F[T]$. Then there exists a regular semi-simple element $\wt{γ}\in G'_{\mr{rs}}$ with $\Inv(\wt{γ}; T) = π^{-n}\Inv(γ; πT)$ and
$$\Orb(γ,f'_{1/2}, s) = \Orb(\wt{γ}, f'_0, s).$$
\item Assume that $g\in G_{1/2, \mr{rs}}$ and that $π^{-n}\Inv(g;πT) \in T^n + πO_F[T]$. Then there exists a regular semi-simple $\wt g\in G_{0, \mr{rs}}$ with $\Inv(\wt g; T) = π^{-n}\Inv(g; πT)$ and
$$\Orb(g, f_{1/2}) = \Orb(\wt g, f_0)).$$
\item Assume that $g\in G_{1/2, \mbY, \rs}$ is regular semi-simple with $π^{-n}\Inv(g;πT) \in T^n + πO_F[T]$. Then there exists a regular semi-simple $\wt g\in G_{0, \mbY, \mr{rs}}$ with $\Inv(\wt g; T) = π^{-n}\Inv(g; πT)$ and
$$\mr{Int}(g) = \mr{Int}(\wt g).$$
\end{enumerate}
\end{thm}

The proof of Theorem \ref{thm:auxiliary} will be constructive in the sense that we work with explicit orbit representatives $γ$ resp. $g$ and then define concrete elements $\wt{γ}$ and $\wt g$.

The condition on the invariants of $γ \in G'$ resp. $g\in G_{1/2}$ or $g\in G_{1/2, \mbY}$ in Theorem \ref{thm:auxiliary} can also be formulated by saying that $z_γ^2/π \in L_γ$ resp. $z_g^2/π\in L_g$ are topologically nilpotent. It is clear that Theorem \ref{thm:auxiliary} immediately implies Theorems \ref{thm:FL} and \ref{thm:main} for such elements. Moreover, the cases where $z_γ^2/π$ or $z_g^2/π$ are not integral over $\O_F$ are trivial: Here, all orbital integrals or intersection numbers in question vanish, a fact proven in Corollary \ref{vanish_par}, Lemma \ref{vanish_b} and Lemma \ref{lem:integrality_y}. This leaves the edge cases where $z_γ^2/π$ or $z_g^2/π$ are integral but not topologically nilpotent. We will prove some auxiliary results (Corollary \ref{reduction}, Lemma \ref{reduction_b}, Proposition \ref{prop:factorization_geometric}) that allow to split off the non-topologically nilpotent part and to treat it separately.

\section{Fundamental Lemma}
\label{s:FL}
In this section we prove parts (1) and (2) of Theorem \ref{thm:auxiliary}. As a corollary we will obtain the fundamental lemma for $GL_n(B)$ in Theorem \ref{thm:main}. Our main tool is the combinatorial interpretation of orbital integrals in terms of lattice counts.

\subsection{Orbital integrals on $G^\prime$}
\label{ss:Gprime}
Our aim in this section is to relate the orbital integrals $\Orb(γ,f^\prime_{1/2}, s)$ and $\Orb(γ,f^\prime_0, s)$ for regular semi-simple $\gamma \in G^\prime_{\rs}$. They both only depend on the double coset $H'γH'$. Hence it suffices to consider elements of the form $\gamma(x) = \begin{psmallmatrix}
1 & x \\
1 & 1
\end{psmallmatrix}$, where $x\in GL_n(F)$ and where $1 \in GL_n(F)$ denotes the unit element. Recall that we define $z_γ = γ_+^{-1}γ_-$ and $L_\gamma = F[z^2_\gamma]$. Note that $z_{\gamma(x)} = \begin{psmallmatrix}
0 & x \\
1 & 0
\end{psmallmatrix}$ and $z^2_{\gamma(x)} = \begin{psmallmatrix}
x & 0 \\
0 & x
\end{psmallmatrix}$. So it is by definition that the invariant polynomial $\Inv(\gamma(x),T)$ of $γ(x)$ equals the characteristic polynomial of $x$. We next recall some definitions and results from \cite{LM2}, but specialized to our quaternion algebra setting. Let $K = F\times F$ be as in \S\ref{ss:setting}.

\defi{[\protect{\cite[Definition 3.20]{LM2}}] Let $\mcL$ denote the set of $O_K$-lattices in $F^{2n}$. For $γ\in G'_\rs$ regular semi-simple and $λ \in \{0, 1/2\}$, we define $\mcL_λ(γ)$ in the following way. For $λ = 0$, we put
\eqn{\sL_0(\gamma):= \{ \Lambda \in \sL \mid \gamma \Lambda \in \sL \}.}
For $λ = 1/2$, and where $(π, 1)\in O_K$ is the element $\mr{diag}(π, 1)$, we put
\eqn{\sL_{1/2}(\gamma):= \{ Λ\in \sL \mid z_γ^2Λ\subseteq z_γ(π,1)Λ \subseteq πΛ\}.}
These two sets equal the set $\mcL(γ)$ from \cite[Definition 3.20]{LM2} in the two cases of Hasse invariant $0$ and $1/2$. In the case of Hasse invariant $1/2$, we have furthermore used the equivalent description in \cite[(3.26)]{LM2}.
}
Note that the $O_K$-lattices $Λ\in \mcL$ are precisely the direct sums $Λ_+\oplus Λ_-$ of $O_F$-lattices $Λ_+, Λ_-\subset F^n$. We use this to give a more concrete description of $\mcL_{1/2}(γ(x))$:
\defi{For $x\in GL_n(F)$, define
\eqn{\label{eq:Lx}\sL(x) := \left\{(\Lambda_+, \Lambda_-) \text{ pairs of $O_F$-lattices in $F^n$} \mid x\Lambda_- \subseteq \Lambda_+ \subseteq \Lambda_-\right\}.}
Since $z_{\gamma(x)} = \begin{psmallmatrix}
0 & x \\
1 & 0
\end{psmallmatrix}$, it follows directly from definitions that
$$\begin{array}{rcl}
\mcL(x/π) & \overset{\sim}{\lr} & \mcL_{1/2}(γ(x))\\[4pt]
(Λ_+, Λ_-) & \longmapsto & Λ_+ \oplus Λ_-.
\end{array}$$
}

\defi{[\protect{\cite[Definition 3.20]{LM2}}] \label{def:omega} For a regular semi-simple element $γ\in G'_\rs$ and a lattice $Λ\in \mcL_0(γ)$, we put
\eqn{Ω(γ, Λ, s) := Ω(h^{-1}_1 γh_2 , s),}
where $h_1, h_2 \in H^\prime$ are chosen such that $\Lambda = h_1 \O^{2n}_F$ and $\gamma\Lambda = h_2 \O^{2n}_F$.
}

For $x\in GL_n(F)$, we let $L_x = F[x] \subseteq M_n(F)$ be the generated $F$-algebra. For a pair $\Lambda = (\Lambda_+, \Lambda_-) \in \sL(x)$, we denote by $R_{\Lambda} \subset L_x$ the stabilizer of $\Lambda$ under the diagonal action of $L_x$ on lattices. We can now state the combinatorial interpretation of orbital integrals from \cite[Equation (3.25)]{LM2} of the test function $f'_{1/2}$:
\eqn{\label{eq:orb_formula}\Orb(γ(x),f^\prime_{1/2}, s) = q^{-ns}\sum_{\Lambda = (\Lambda_+,\Lambda_-) \in \sL(x/\pi)/L_x^\times} 
[\O_{L_x}^{\times}:R^{\times}_{\Lambda}]\,\, Ω(γ(x),\, Λ_+ \oplus \Lambda_-,\, s).}

\begin{cor}\label{vanish_par}
If $x/\pi$ is not integral over $O_F$, then the orbital integral $\Orb(γ(x),f^\prime_{1/2},s)$ vanishes identically.
\end{cor}
\begin{proof}
The set $\mcL(x/π)$ is empty if $x/π$ is not integral.
\end{proof}

Assume now that $x$ is topologically nilpotent. Then \cite[Lemma 3.21 (1)]{LM2} states that there is a bijection
$$\begin{array}{rcl}
\mcL(x) & \overset{\sim}{\lr} & \mcL_0(γ(x))\\[4pt]
(Λ_+, Λ_-) & \longmapsto & Λ_+ \oplus Λ_-.
\end{array}$$
Identity \cite[Equation (3.25)]{LM2} then specializes to
\eqn{\label{eq:orb_formulasph}\Orb(γ(x),1_{GL_{2n}(\O_F)}, s) = \sum_{\Lambda = (\Lambda_+,\Lambda_-) \in \sL(x)/L_x^\times} 
[\O_{L_x}^{\times}:R^{\times}_{\Lambda}]\,\, Ω(γ(x),\, Λ_+ \oplus \Lambda_-,\, s).}

Combining \eqref{eq:orb_formula} and \eqref{eq:orb_formulasph} yields a proof of Theorem \ref{thm:auxiliary} (1):

\begin{proof}[Proof of Theorem \ref{thm:auxiliary} (1)]
First recall the statement: Let $γ\in G'_\rs$ be regular semi-simple. The statement we would like to prove only depends on the orbit of $γ$, so we may assume without loss of generality that $\gamma = \gamma(x)$ for some $x\in GL_n(F)$. Set $\wt x = x/π$. The assumption is
$$\mr{char}(\wt x; T) = π^{-n}\Inv(γ(x), πT) \in T^n + π O_F[T]$$
which implies that $\wt{x}$ is topologically nilpotent. In particular, $\wt{\gamma} := \gamma(\wt{x})$ lies in $GL_{2n}(F)$ and is regular semi-simple. It satisfies $\Inv(\wt{γ}; T) = π^{-n}\Inv(γ; πT)$ and our task is to show that
\eqn{\label{eq:id1}\Orb(γ,f^\prime_{1/2},s) =\Orb(\wt{γ},1_{GL_{2n}(\O_F)},s).}
To this end, we compare the two quantities $Ω(γ, Λ_+ \oplus \Lambda_-, s)$ and $Ω(\wt{\gamma}, Λ_+ \oplus \Lambda_-, s)$ for $(\Lambda_+, \Lambda_-) \in \mcL(\wt x)$. Choose $h_1$ and $h_2$ in $H'$ such that $h_1\O^{2n}_F = Λ_+ \oplus \Lambda_-$ and $h_2\O^{2n}_F = \gamma(Λ_+ \oplus \Lambda_-)$. Write $h^{-1}_1 = \begin{psmallmatrix}
a &  \\
 & b
\end{psmallmatrix}$ and $h_2 = \begin{psmallmatrix}
c &  \\
 & d
\end{psmallmatrix}$. We obtain from Definition \ref{def:omega} that
\eqn{\begin{aligned}
Ω(\wt{γ},\, Λ_+ \oplus \Lambda_-,\, s) & = \Omega\left(\begin{pmatrix}
ac & axd/π \\
bc & bd
\end{pmatrix},\, s\right)\\
& = \Omega\left(\begin{pmatrix}
ac & axd \\
bc & bd
\end{pmatrix},\, s\right)q^{-ns} = Ω(γ,\, Λ_+ \oplus \Lambda_-,\, s)q^{-ns}.
\end{aligned}}
Substituting this into \eqref{eq:orb_formulasph} for $\wt{\gamma}$ yields \eqref{eq:orb_formula} for $\gamma$ which proves \eqref{eq:id1} as desired.
\end{proof}

Note that if $(\Lambda_+, \Lambda_-) \in \sL(x)$, then both $\Lambda_+$ and $\Lambda_-$ are $O_F[x]$-modules. For this reason, we next focus on the ring $O_F[x]$, in particular on its idempotents.

\lemm{\label{lemm: hensel} Assume that $x\in GL_n(F)$ is integral over $O_F$. Then there is a unique way to write $\O_F[x]$ as a product $R_0 \times R_1$ such that the image $(x_0, x_1)$ of $x$ has the property that $x_0$ is topologically nilpotent and $x_1 \in R_1^\times$ a unit.}
\pf{As $x$ was assumed to be integral over $\O_F$, the ring $O_F[x]$ is of the form $\O_F[T]/(P(T))$ for some monic polynomial $P(T) \in \O_F[T]$. We consider the reduction $\overbar{P}$ of $P$ modulo $\pi$. It will factor as $\overbar{P} = T^m f(T)$ with $f(0) \neq 0$. By Hensel's lemma, this factorization lifts to a factorization of $P$ which defines the desired factorization $R_0\times R_1$.
}

\begin{cor}\label{reduction}
Let $x\in GL_n(F)$ have the property that $\wt x = x/π$ is integral over $O_F$. Let $O_F[\wt x] = R_0 \times R_1$ be the factorization from Lemma \ref{lemm: hensel} with respect to $\wt x$ and let $(x_0, x_1)$ denote the components of $x$. Let $F^n = V_0\times V_1$ be the corresponding factorization of $F^n$ and fix isomorphisms $V_i \iso F^{n_i}$. Then $γ = γ(x)$ lies in $G'_\rs$ while $γ_0 = γ(x_0)$ and $γ_1 = γ(x_1)$ lie in $GL_{2n_i}(F)_\rs$. There is an identity of orbital integrals 
\eqn{\Orb(γ,f^\prime_{1/2},s) = \Orb(γ_0,f^\prime_{1/2, n_0},s)\Orb(γ_1,f^\prime_{1/2, n_1},s)}
where the test functions on the right hand side are meant in the sense of \eqref{eq:test_functions_1_2} but on $GL_{2n_i}(F)$.
\end{cor}
\pf{
Every $O_F[\wt x]$-lattice is a direct sum of an $R_0$-lattice and an $R_1$-lattice. It is easily seen that this defines a bijection $\sL(\wt x)\overset{\sim}{\to} \sL(x_0/π) \times \sL(x_1/π)$. Furthermore, the definition of $Ω(γ, Λ_+ \oplus \Lambda_-, s)$ is multiplicative in such direct sums. The desired identity then follows from \eqref{eq:orb_formula}.
}

Identity \eqref{eq:id1} already covers the factor $\Orb(γ_0, f'_{1/2, n_0},s)$, so we now turn to the factor $\Orb(γ_1, f'_{1/2, n_1}, s)$. For elements $x\in GL_n(F)$ that are regular semi-simple in the usual sense, we consider the conjugation orbital integral
\eqn{\label{eq:orb_int} \Orb(x,1_{GL_n(\O_F)}) = \int_{GL_n(F)/L_x^\times} 1_{GL_n(O_F)} ( y^{-1}xy)dy,
}
where $L_x = F[x]$ as before and where $dy$ is the Haar measure on $GL_n(F)/L_x^\times$ that is normalized by
\eqn{\Vol(GL_n(\O_F)) = \Vol(\O^\times_{L_x}) = 1.}

\lemm{\label{aux1lattice} Assume that $x\in GL_n(F)$ has the property that $\wt x = x/π$ is integral over $O_F$ with $\det(\wt x) \in O_F^\times$. Then there is an identity of orbital integrals 
\eqn{\Orb(γ(x),f^\prime_{1/2},s) = \Orb(\wt{x},1_{GL_n(\O_F)}).}}
\pf{
We again use \eqref{eq:orb_formula} to express the left hand side. A general identity (see \cite[Lemma 3.23]{LM2}) states that the transfer factor of a lattice $Λ = Λ_+\oplus Λ_-\in \mcL_0(γ)$ is given by
\eqn{\label{eq:tr_fact} Ω(γ, Λ_+ \oplus \Lambda_-, s) = (-1)^{[(\gamma\Lambda)_-:z\Lambda_+] + [(\gamma\Lambda)_-:\Lambda_-]}q^{([(\gamma\Lambda)_+:z\Lambda_-] + [(\gamma\Lambda)_-:z\Lambda_+])s},}
where $z = z_γ$ and where $[Λ_1:Λ_2]$ denotes the length of $Λ_1/Λ_2$. In the situation of \eqref{eq:orb_formula}, we apply this formula to the element $γ = γ(x)$ and a lattice $Λ = Λ_+ \oplus Λ_-$ with $(Λ_+, Λ_-) \in \mcL(\wt x)$. We have that $z = \begin{psmallmatrix}
0 & x \\
1 & 0
\end{psmallmatrix}$ and that $γΛ = Λ$ because $x$ is topologically nilpotent under the assumption $\det(\wt x) \in O_F^\times$. The assumption moreover implies that any $\mcL(\wt x)$ is the set of lattice pairs $(Λ_+, Λ_-)$ that satisfy
$$\wt x Λ_+ = Λ_- = Λ_+.$$
We see that $z\Lambda_+ = \Lambda_-$ and $zΛ_- = πΛ_+$ for all such $(Λ_+, Λ_-)$. Substituting this in \eqref{eq:tr_fact} gives $Ω(γ(x), Λ_+ \oplus \Lambda_-, s) = q^{ns}$ for all $(Λ_+, Λ_-) \in \sL(\wt x)$ and hence
\eqn{\Orb(γ,f^\prime_{1/2},s) = q^{ns}q^{-ns}\sum_{\Lambda = (\Lambda_+,\Lambda_-) \in \sL(\wt x)} 
[\O_{L_x}^{\times}:R^{\times}_{\Lambda}].
}
This is precisely the combinatorial description of $\Orb(\wt{x},1_{GL_n(\O_F)})$ and the proof is complete.
}

\subsection{Orbital integrals on $G_\lambda$}
\label{ss:G}
Recall that $E/F$ is an unramified quadratic extension and that $B_λ$ denotes a quaternion algebra over $F$ of Hasse invariant $λ\in \{0,1/2\}$ with an embedding $E\to B_λ$. Recall that $σ\in \Gal(E/F)$ denotes the non-trivial element. Our aim in this section is to relate the orbital integrals $\Orb(-, f_0)$ and $\Orb(-, f_{1/2})$, where $f_λ \in C^\infty_c(G_λ)$ is the characteristic function of $GL_n(O_{B_λ})$. Put $ε = 2λ \in \{0, 1\}$ and fix an element $\varpi \in O_{B_λ}$ that satisfies
\begin{equation}\label{uniformizer_of_quaternions}
    \varpi^2 = π^ε\quand \varpi a = σ(a)\varpi\ \ \text{for $a\in E$}.
\end{equation}
For $x\in GL_n(E)$, we define $g(x) = 1 + x \varpi \in M_n(B_λ)$. This element lies in $G_{λ, \rs}$ if and only if the characteristic polynomial of $z_{g(x)}^2 = xσ(x)π^ε \in GL_n(E)$ is separable and does not vanish at $0$ or $1$. In this case, $\Inv(g(x);T) = \mr{char}(xσ(x)π^ε;T)$. Since $\Orb(g, f_λ)$ only depends on the double coset $HgH$, we may restrict attention to group elements of the form $g(x)$.

\defi{\label{L_b} Denote by $σ:E^n\to E^n$ the coordinate-wise Galois conjugation. For $x\in GL_n(E)$, define $\sL^\sigma(x)$ as the set of $\O_E$-lattices $Λ\subset E^n$ that satisfy
\eqn{x σ(\Lambda) \subseteq \Lambda.}
Moreover, define $L_x \subset M_n(E)$ as the subalgebra $F[xσ(x)π^ε]$. Then $L_x^\times$ acts by multiplication on $\mcL^σ(x)$.
}

\lemm{\label{aux2quat} Let $λ\in \{0, 1/2\}$ be any and let $B = B_λ$. Let $g = 1 + x \varpi \in G_{λ,\rs}$ be regular semi-simple. Suppose further that $xσ(x)\pi^{\varepsilon}$ is topologically nilpotent. Then there is the identity
\eqn{\label{eq:orb55} \Orb(g(x),1_{GL_n(O_B)}) = \sum_{\Lambda \in \sL^\sigma(x)/L_x^\times} 
[O_{L_x}^{\times}:R^{\times}_{\Lambda}].}
Here $R_{\Lambda} \subseteq L_x$ denotes the order that stabilizes $\Lambda$. 
}
\pf{The assumption that $xσ(x)π^ε=(x\varpi)^2$ is topologically nilpotent implies that the determinant $\det(g(x))$ lies in $O_F^\times$. (Here and in the following, the determinant is meant in the sense of the reduced norm $GL_n(B)\to F^\times$.) Hence the condition $h_1^{-1}g h_2 \in GL_n(O_B)$ holds if and only if $h_1^{-1}g h_2 \in M_n(O_B)$. Since $h_1^{-1}g h_2 = h_1^{-1} h_2(1 + h_2^{-1}x\varpi h_2)$ and since $h_2^{-1}x\varpi h_2$ is topologically nilpotent by assumption, this is equivalent to $h_1^{-1} h_2 \in GL_n(O_E)$ and $h_2^{-1} x  σ(h_2) \in M_n(O_E)$. Given a pair $(h_1, h_2)$ with $h_1^{-1}h_2 \in GL_n(O_E)$, consider the lattice $\Lambda = h_1 O_E^n= h_2 O_E^n$. Then $h_2^{-1} x σ(h_2)$ lies in $M_n(O_E)$ if and only if $xσ(Λ) \subseteq \Lambda$. Rewriting the definition of $\Orb(g, 1_{GL_n(O_B)})$ in this way gives \eqref{eq:orb55}.
}

\begin{lem}\label{vanish_b}
Let $g \in G_{1/2,\rs}$ and $B = B_{1/2}$. The orbital integral $\Orb(g, 1_{GL_n(O_B)})$ vanishes if $z^2_g/\pi$ is not integral over $O_F$.
\end{lem}
\pf{Since $\Orb(g, 1_{GL_n(O_B)})$ only depends on the orbit of $g$ and since every regular semi-simple orbit contains a representative of the form $g(x)$ we may assume without loss of generality that $g = g(x)$ for some $x\in GL_n(E)$. In this case $z^2_g/\pi = xσ(x)$. Moreover, $GL_n(O_B) = GL_n(O_E) + M_n(O_E)\varpi$ because $B$ is the division algebra. So $h_1^{-1} h_2 + h_1^{-1} x σ(h_2) \varpi \in GL_n(O_B)$ with $h_i \in GL_n(E)$ can only hold if $h_1^{-1} h_2 \in GL_n(O_E)$ and $h_2^{-1} xσ(h_2) \in M_n(O_E)$. The second condition implies that $xσ(x)$ is integral over $O_F$ because it is conjugate to $h_2^{-1}xσ(h_2) σ(h_2^{-1}xσ(h_2))$. This was to be shown.}

\begin{proof}[Proof of Theorem \ref{thm:auxiliary} (2)]
Let $g\in G_{1/2, \rs}$ be such that $π^{-n}\Inv(g; πT) \in T^n + πO_F[T]$. Denote by $\varpi_{λ}\in B_λ$ the element fixed in \eqref{uniformizer_of_quaternions}. Without loss of generality we may assume that $g = 1 + x \varpi_{1/2}$ for some $x\in GL_n(E)$. The assumption on $g$ is then equivalent to $xσ(x)$ being topologically nilpotent. The element $\wt g = 1 + x \varpi_{0}$ hence lies in $G_{0, \rs}$ and satisfies $\Inv(\wt g;T) = π^{-n}\Inv(g; πT)$. Lemma \ref{aux2quat} applies to both $g$ and $\wt g$, and yields
\begin{equation}\label{eq:orb56}
\Orb(g, 1_{GL_n(O_B)}) = \Orb(\wt g, 1_{GL_{2n}(O_F)})
\end{equation}
as desired.
\end{proof}

We next prove a factorization of the orbital integral that is analogous to that in Lemma \ref{reduction}.

\lemm{\label{reduction_b} Let $B = B_{1/2}$ and assume that $g(x) \in G_{1/2,\rs}$ is such that $xσ(x)$ is integral over $O_F$. Let $O_F[xσ(x)] = R_0\times R_1$ be as in Lemma \ref{lemm: hensel}. Let $E^n = V_0\times V_1$ be the induced decomposition of $E^n$ which is preserved by $x$. Write $(x_0, x_1)$ for its components and choose isomorphisms $V_i \iso E^{n_i}$. Put $g_i = 1 + x_i\varpi\in GL_{n_i}(B)$. Then there is an identity of orbital integrals 
\eqn{\Orb(g,1_{GL_n(O_B)})=\Orb(g_0,1_{GL_{n_0}(O_B)})\Orb(g_1,1_{GL_{n_1}(O_B)}).}
}
\pf{
Every $O_F[xσ(x)]$-lattice is a direct sum of an $R_0$-lattice and an $R_1$-lattice. This defines a bijection $\sL^\sigma(x)\overset{\sim}{\to} \mcL^σ(x_0)\times \mcL^σ(x_1)$. The desired identity now follows from Lemma \ref{aux2quat} which can be applied to all three orbital integrals because $xσ(x)$, $x_0σ(x_0)$ and $x_1σ(x_1)$ are all integral by assumption.
}

Identity \eqref{eq:orb56} already covers the factor $\Orb(g_0, 1_{GL_{n_0}(O_B)})$, so we now turn to the factor $\Orb(g_1, 1_{GL_{n_1}(O_B)})$. Let $x \in GL_n(O_E)$ be an element that is regular semi-simple with respect to the $σ$-twisted conjugation action $(y,x) \mapsto y^{-1} x σ(y)$. Its stabilizer then equals $L^\times_x$ where $L_x = F[xσ(x)]$. We normalize the Haar measure $dy$ on $GL_n(E)/L_x^\times$ by
\eqn{\Vol(GL_n(O_E)) = \Vol(\O^\times_{L_x}) = 1}
and define a twisted orbital integral by
\begin{equation}\label{eq:twisted_orb_int}
\Orb^\sigma(x,1_{GL_n(O_E)}) := \int_{GL_n(E)/L_x^\times} 1_{GL_n(O_E)} ( y^{-1} x σ(y)) dy.
\end{equation}

\lemm{\label{aux2lattice}Let $B = B_{1/2}$ and let $x\in GL_n(E)$ be such that $xσ(x)$ is integral over $O_F$ with $\det(xσ(x))\in O_F^\times$. Then there is the identity
\begin{equation}\label{eq:twisted_comp}
\Orb(g(x), 1_{GL_n(O_B)}) = \Orb^\sigma(x,1_{GL_n(O_E)}).
\end{equation}
}
\pf{The assumption that $xσ(x)$ is integrally invertible implies that $\mcL^σ(x)$ is the set of $O_E$-lattices $Λ\subseteq E^n$ such that $xσ(Λ) = Λ$. In this case, the lattice counting expression in Lemma \ref{aux2quat} equals $\Orb^\sigma(x,1_{GL_n(O_E)})$.
}

\subsection{Vanishing orders}
\label{ss:vanishing}
The FL and the AFL both include vanishing statements. In order to prove these, we now recall some results from \cite{LM1} for the case $λ = 0$.

By regular semi-simple invariant (of degree $n$), we mean a degree $n$ polynomial $δ\in F[T]$ that is monic, separable and satisfies $δ(1)δ(0)\neq 0$. These are precisely the polynomials that arise as invariant polynomials of elements $γ\in G'_\rs$. Given a regular semi-simple $δ\in F[T]$, we define
\begin{equation}\label{eq:L_B}
L_δ := F[z^2]/(δ(z^2))\quand B_δ := (E\tensor_F L_δ)[z]/(z(a\tensor b) = (σ(a)\tensor b)z)_{a\in E,\ b\in L_δ}.
\end{equation}
Then $L_δ$ is an étale $F$-algebra of degree $n = \deg(δ)$ by separability, and $B_δ/L_δ$ is a quaternion algebra. Moreover, $B_δ$ contains $E$ by construction. Let $L_δ = \prod_{i\in I} L_i$ and $B_δ = \prod_{i\in I} B_i$ be the factorizations of $L_δ$ and $B_δ$ according to the idempotents in $L_δ$. The algebraic vanishing order of $δ$ from \cite[Definition 2.15]{LM1} is defined as the integer
$$\ord_0(δ) := \#\{i\in I\mid \text{$B_i$ is a division algebra}\}.$$
One of the main results of \cite{LM1} is a factorization formula for $\Orb(γ, 1_{GL_{2n}(O_F)}, s)$. It implies, see \cite[Corollary 2.17]{LM1}, that for every regular semi-simple $γ\in G'_{\mr{rs}}$,
\begin{equation}\label{eq:vanishing_0}
\ord_{s = 0} \Orb(γ, 1_{GL_{2n}}(O_F), s) \geq \ord_0(\Inv(γ)).
\end{equation}
Our aim is to formulate an analogous result in the case of invariant $λ = 1/2$.
\begin{defn}\label{ord_1/2}
Let $δ\in F[T]$ be a regular semi-simple invariant of degree $n$ and let $L_δ$, $B_δ$ be defined as in \eqref{eq:L_B}. Let $B = B_{1/2}$ be the quaternion division algebra over $F$. Define
$$\ord_{1/2}(δ) := \#\{i\in I \mid B_i \not\iso B\tensor_F L_i\}.$$
In other words, $\ord_{1/2}(δ)$ is the number of indices such that $B_i$ is split and $[L_i:F]$ odd, or such that $B_i$ is division and $[L_i:F]$ even. It is checked with a simple case distinction that
\begin{equation}\label{eq:rel_alg_orders}
\ord_{1/2}(δ(T)) = \ord_{0}(π^{-n}δ(πT)).
\end{equation}
\end{defn}
\begin{cor}\label{cor:vanishing}
Let $γ\in G'_{\mr{rs}}$ be a regular semi-simple element. Then
\begin{equation}\label{eq:vanishing}
\ord_{s = 0} \Orb(γ, f'_{1/2}, s) \geq \ord_{1/2}(\Inv(γ)).
\end{equation}
\end{cor}
\begin{proof}
If $z_γ^2/π$ is not integral over $O_F$, then $\Orb(γ, f'_{1/2}, s) = 0$ by Corollary \ref{vanish_par} and there is nothing to prove. So assume that $z^2_γ/π$ is integral and consider the orbital integral factorization from Corollary \ref{reduction},
$$\Orb(γ, f'_{1/2}, s) = \Orb(γ_0, f'_{1/2}, s)\Orb(γ_1, f'_{1/2}, s).$$
Since $\Inv(γ) = \Inv(γ_0)\Inv(γ_1)$, we have
$$\ord_{1/2}(\Inv(γ)) = \ord_{1/2}(\Inv(γ_0)) + \ord_{1/2}(\Inv(γ_1)).$$
It hence suffices to show \eqref{eq:vanishing} for $γ_0$ and $γ_1$ separately, meaning we may either assume $z_γ^2/π$ to be topologically nilpotent or integrally invertible.

If $z_γ^2/π$ is topologically nilpotent, then we can apply Theorem \ref{thm:auxiliary} (1): For any element $\wt{γ} \in G'_\rs$ with $\Inv(\wt{γ}; T) = π^{-n}\Inv(γ; πT)$, we have $\Orb(γ, f'_{1/2}, s) = \Orb(\wt{γ}, 1_{GL_{2n}(O_F)}, s)$. Using \eqref{eq:rel_alg_orders}, the desired vanishing statement \eqref{eq:vanishing} follows directly from \eqref{eq:vanishing_0}.

If $z_γ^2/π$ is integral over $O_F$ and $\det(z_γ^2/π) \in O_F^\times$, then
$$\ord_{1/2}(\Inv(γ; T)) = \ord_0(π^{-n}\Inv(γ; πT)) = 0$$
because $z^2_γ/π$ then lies in $O_{L_γ}^\times$ and is hence a norm from $E\tensor_F L_γ$. (This uses that $E/F$ is unramified.) The inequality $\ord_{s = 0} \Orb(γ, f'_{1/2}, s) \geq 0$ holds trivially and there is nothing to prove in this case. This completes the argument. (We remark that if $z_γ^2/π\in O_{L_γ}^\times$, then Corollary \ref{aux1lattice} states that $\Orb(γ, f'_{1/2}, s)$ is a constant independent of $s$.)
\end{proof}

\subsection{Proof of Theorem \ref{thm:FL}}
\label{ss:proof_FL}
We can finally deduce the fundamental lemma for $λ = 1/2$.

\begin{proof}[Proof of Theorem \ref{thm:FL}]
Let $\gamma \in G^\prime_{\rs}$. We need to show the identity
\begin{equation}
\label{*}\Orb(γ, f'_{1/2}) = \begin{cases} \Orb(g, f_{1/2}) & \text{if there exists a matching $g\in G_{1/2}$}\\
0 & \text{otherwise.}
\end{cases}
\end{equation}
We first note that if $π^{-n}\Inv(γ;πT) \notin O_F[T]$, then both sides of the equation vanish. This follows from Corollary \ref{vanish_par} and from Lemma \ref{vanish_b}, respectively. Hence we assume from now on that $π^{-n}\Inv(γ;πT) \in O_F[T]$ or, equivalently, that $z^2_\gamma/\pi$ is integral. Consider $L_δ$ and $B_δ$ as in Equation \eqref{eq:L_B} for $\delta = \Inv(γ;T)$. Let $L_δ=\prod_i L_i$ and $B_δ=\prod_i L_i$ be their factorizations according to idempotents. By \cite[Corollary 2.8]{LM2} there exists an element $g \in G_{1/2, \rs}$ that matches $\gamma$ if and only if $\ord_{1/2}(δ) = 0$, see Definition \ref{ord_1/2}. In particular, if there is no matching element $g \in G_{1/2, \rs}$, then $\Orb(γ, f'_{1/2}) = 0$ by Corollary \ref{cor:vanishing}. This proves the vanishing part of \eqref{*}.

We henceforth consider the case that there is an element $g \in G_{1/2, \rs}$ that matches $\gamma$. Assuming that $γ = 1 + z_γ$ and that $g = 1 + z_g$, let $γ = (γ_0, γ_1)$ and $g = (g_0, g_1)$ be the components of $γ$ and $g$ such that $z^2_{γ_0}/π$ and $z^2_{g_0}/π$ are topologically nilpotent, and $z_{γ_1}^2/π$ and $z_{g_1}^2/π$ integrally invertible (Corollaries \ref{aux1lattice} and \ref{aux2lattice}). Then $γ_0$ matches $g_0$ and $γ_1$ matches $g_1$ as can be seen by using the isomorphism $O_F[z^2_\gamma]\cong O_F[z^2_g]$ that sends $z^2_\gamma$ to $z^2_g$. By Lemmas \ref{reduction} and \ref{reduction_b}, the two sides of \eqref{*} factor and the desired equality becomes
$$\Orb(γ_0, f'_{1/2, n_0})\Orb(γ_1, f'_{1/2, n_1}) \overset{?}{=} \Orb(g_0, 1_{GL_{n_0}(O_B)}) \Orb(g_1, 1_{GL_{n_1}(O_B)}).$$
Here, $B = B_{1/2}$. We can prove this identity factor-by-factor, meaning we may assume that $z_γ^2/π$ is topologically nilpotent or that $z_γ^2/π \in O_{L_γ}^\times$.

Assume first that $z_γ^2/π$ is topologically nilpotent. Then we may apply Theorem \ref{thm:auxiliary} (1) and (2): Let $\wt {γ} \in G'_{\mr{rs}}$ and $\wt g \in G_{0, \rs}$ be such that $\Inv(\wt{γ}; T) = π^{-n}δ(πT) = \Inv(\wt g; T)$. We obtain
$$\Orb(γ, f'_{1/2}) = \Orb(\wt{γ}, 1_{GL_{2n}(O_F)}) = \Orb(\wt g, 1_{GL_{2n}(O_F)}) = \Orb(g, 1_{GL_n(O_B)})$$
where the middle equality is the Guo--Jacquet FL, i.e. the case $λ = 0$ of Theorem \ref{thm:FL}.

Assume now that $z_γ^2/π \in O_{L_γ}^\times$. Without loss of generality, we may assume that $γ = γ(x)$ for some $x\in GL_n(F)$ and $g = 1 + x'\varpi$ for some $x'\in GL_n(E)$. The fact that $γ$ and $g$ match translates to the identity
$$\mr{char}_F(x/π; T) = \mr{char}_E(x'σ(x'); T).$$
In other words, $x/π$ and $x'$ match in the sense of the quadratic base change FL, see \cite{K}. We obtain from Lemmas \ref{aux1lattice} and \ref{aux2lattice}, as well as the base change FL that
$$\Orb(γ, f'_{1/2}) = \Orb(x/π, 1_{GL_n(O_F)}) = \Orb^σ(x', 1_{GL_n(O_E)}) = \Orb(g, 1_{GL_n(O_B)}).$$
The proof of Theorem \ref{thm:FL} is now complete.
\end{proof}

%alternatively \cite[Proposition 4.9.]{AC}

\section{Arithmetic Fundamental Lemma}
\label{s:AFL}
It is left to prove Theorem \ref{thm:auxiliary} (3) and Theorem \ref{thm:main} which is the aim of this final section. Its structure is analogous to that of \S\ref{s:FL}: We first relate the intersection problems for $λ = 0$ and $λ = 1/2$ by analyzing their moduli descriptions (\S\ref{ss:moduli}). Then we prove a factorization result for $\mr{Int}(g)$ in the case of $λ = 1/2$ (\S\ref{ss:factorization}). Combining both techniques we obtain the proof of Theorem \ref{thm:main} (\S\ref{ss:proof_AFL}). Our notation in this section is the same as in \S\ref{ss:def_intersection}. 

\subsection{Description of $\mcI(g)$}
\label{ss:moduli}
Let $λ \in \{0,1/2\}$ and let $g \in G_{λ, \mbY, \mr{rs}}$ be a regular semi-simple element. Our first aim is to give a more explicit description of $\mcI(g)$. Let $B = B_λ$ be the quaternion algebra of invariant $λ$. Let $\varpi\in B^\times$ be the element from \eqref{uniformizer_of_quaternions}. Recall that this means that $\varpi$ is chosen such that $\varpi a = σ(a)\varpi$ for all $a\in E$ and such that $\varpi^2 = π^ε$ where $ε = 2λ$. Then $O_B = O_E\oplus \varpi O_E$. So for every strict $O_E$-module $\mbY$, we obtain the coordinates $\mbX = \mbY \oplus \varpi \mbY$. With respect to this decomposition, the $O_B$-action $κ:O_B\to \End(\mbX)$ is given by
$$a + b\varpi \longmapsto \begin{pmatrix}
a & b π^ε\\
σ(b) & σ(a)
\end{pmatrix},\quad a,b\in E.$$
The $O_B$-linear endomorphisms of $\mbX$ then have the presentation
\begin{equation}\label{eq:endo_ring_D_explicit}
D_{λ, \mbY} = \left.\left\{\left(\begin{matrix} x & π^εy\\ y & x\end{matrix}\right)\,\right\vert\, x \in C_\mbY,\ y\in \End^0_F(\mbY)\ \mr{s.th.}\ ya = σ(a)y\ \text{for}\ a\in E\right\}.
\end{equation}
Let $g = \begin{psmallmatrix} x & π^εy \\ y & x\end{psmallmatrix} \in G_{λ, \mbY} = D_{λ, \mbY}^\times$ be an element such that both $x$ and $y$ are invertible. Then $z_g$ takes the form
\begin{equation}\label{eq:z_g}
z_g = \begin{pmatrix}
 & π^εx^{-1}y \\ x^{-1}y & 
\end{pmatrix}.
\end{equation}
It follows from this that the invariant polynomial of $g$ is
$$\Inv(g;T) = \mr{charred}_{C_{\mbY}/E}\left(π^ε(x^{-1}y)^2; T\right).$$
\begin{defn}\label{def:Zz}
For an element $z\in \End^0_F(\mbX)$, we denote by $\mcZ(z)\subseteq \mcM_λ$ the closed formal subscheme with functor of points description
$$\mcZ(z)(S) = \{(X, κ, ρ) \in \mcM_λ(S) \mid ρzρ^{-1} \in \End(X)\}.$$
We analogously define $\mcZ(w)\subseteq \mcN$ for an endomorphism $w\in \End^0_F(\mbY)$.
\end{defn}
\begin{lem}\label{lem:Ig}
Let $g\in G_{λ, \mbY, \mr{rs}}$ be a regular semi-simple element such that $z_g$ is topologically nilpotent. Then, as closed formal subschemes of $\mcM_λ$,
$$\mcI(g) = \mcN \cap \mcZ(z_g).$$
Furthermore, writing $z_g = \left(\begin{smallmatrix} & π^εw \\ w & \end{smallmatrix} \right)$ as in \eqref{eq:z_g}, there is the following identity of closed formal subschemes of $\mcN$:
$$\mcN \cap \mcZ(z_g) = \mcZ(w).$$
\end{lem}
\begin{proof}
The first identity is a special case of \cite[Proposition 4.23 (2)]{LM2} which we may apply because $z_g$ is topologically nilpotent by assumption. The second identity follows directly from the definitions of $\mcZ(z_g)$ and $\mcZ(w)$.
\end{proof}

\begin{cor}\label{cor:thm_3}
Part (3) of Theorem \ref{thm:auxiliary} holds.
\end{cor}
\begin{proof}
Let $g\in G_{1/2, \mbY, \mr{rs}}$ be an element whose invariant has the property that $π^{-n}\Inv(g; πT)$ lies in $T^n + πO_F[T]$. We need to construct an element $\wt g \in G_{0, \mbY, \mr{rs}}$ such that both $\Inv(\wt g, T) = π^{-n}\Inv(g, πT)$ and $\mr{Int}(g) = \mr{Int}(\wt g)$.

Given $g$, we define $\wt g \in D_{0, \mbY}$ by the following relation:
$$g = \begin{pmatrix}
x & πy \\ y & x
\end{pmatrix}\quand \ \wt g := \begin{pmatrix}
x & y \\ y & x
\end{pmatrix}.$$
Since $g$ is regular semi-simple, $g_+ = \left(\begin{smallmatrix} x & \\ & x\end{smallmatrix} \right)$ and $g_- = \left(\begin{smallmatrix} & πy \\ y & \end{smallmatrix} \right)$ are both invertible. Then $\wt g_+ = g_+$ and $\wt g_- = \left(\begin{smallmatrix} & y \\ y & \end{smallmatrix} \right)$ are invertible as well. The assumption on $\Inv(g;T)$ implies that $\Inv(g;π) \neq 0$, so $\wt g$ lies in $G_{0,\mbY}$ and has invariant polynomial
$$\Inv(\wt g; T) = π^{-n}\Inv(g; πT).$$
This polynomial is separable by assumption on $g$, so $\wt g$ is regular semi-simple. Moreover, $z_{\wt g}^2$ is topologically nilpotent since $\Inv(\wt g; T) \equiv T^n$ modulo $πO_F[T]$. We may now apply Lemma \ref{lem:Ig} twice to see that
\begin{equation}\label{eq:ident_intersection}
\begin{aligned}
\mcI(\wt g) & = \mcN \cap_{\mcM_0} \mcZ\left(\left(\begin{smallmatrix}  & y \\ y & \end{smallmatrix}\right)\right)\\
& = \mcZ(y)\\
& = \mcN \cap_{\mcM_{1/2}} \mcZ\left(\left(\begin{smallmatrix}  & πy \\ y & \end{smallmatrix}\right)\right)\\
& = \mcI(g).
\end{aligned}
\end{equation}
Note that $z_g^2 = πz_{\wt g}^2$, so the two $F$-algebras $L_g = F[z_g^2]$ and $L_{\wt g} = F[z_{\wt g}^2]$ agree as subalgebras of $C_\mbY$. The isomorphism in \eqref{eq:ident_intersection} is then equivariant with respect to the action of $L_g^\times = L_{\wt g}^\times$. Choosing the same subgroup $Γ\subset L_g^\times = L_{\wt g}^\times$ in the definition of the intersection number (Definition \ref{def:intersection}), we obtain $Γ\backslash \mcI(g) = Γ\backslash \mcI(\wt g)$ and hence the identity
$$\mr{Int}(g) = \mr{Int}(\wt g)$$
as was to be shown.
\end{proof}

\subsection{Factorization of $\mr{Int}(g)$}
\label{ss:factorization}
Recall that $\mbY = \mbY^\circ\times \mbY^\et$ and $\mbX = \mbX^\circ\times \mbX^\et$ denote the connected-étale decompositions of $\mbY$ and $\mbX$. We define RZ spaces $\mcN^\circ$, $\mcN^\et$, $\mcM_λ^\circ$ and $\mcM_λ^\et$ in complete analogy to our definitions of $\mcN$ and $\mcM$ in \eqref{eq:def_LT_space}, \eqref{eq:def_mysterious_space}, but using the objects $\mbY^\circ$, $\mbY^\et$, $\mbX^\circ$ and $\mbX^\et$ instead of $\mbY$ and $\mbX$. For $g = (g^\circ, g^\et)\in G_{λ, \mbY, \mr{rs}}$, we can then also define
$$\mcI(g^\circ) := \mcN^\circ \cap g^\circ (\mcN^\circ) \quand \mcI(g^\et) := \mcN^\et \cap g^\et (\mcN^\et)$$
where the intersections happen on $\mcM^\circ_λ$ and $\mcM^\et_λ$. 
\begin{lem}\label{lem:I_factors}
The set $\mcI(g)(\mbF)$ has the product structure
\begin{equation}\label{eq:prod_Ig}
\mcI(g)(\mbF) = \mcI(g^\circ)(\mbF) \times \mcI(g^\et)(\mbF).
\end{equation}
\end{lem}
\begin{proof}
Every strict $O_F$-module over $\mbF$ is, in a unique way, the product of its identity connected component and its maximal étale quotient. This decomposition is functorial in all respects, giving \eqref{eq:prod_Ig} from definitions.
\end{proof}

\begin{lem}\label{lem:integrality_y}
Let $g\in G_{1/2, \mbY, \mr{rs}}$ be regular semi-simple and such that $\mcI(g)\neq \emptyset$. Then $z_g^2/π$ is integral over $O_F$.
\end{lem}
\begin{proof}
By the product structure on $\mcI(g)(\mbF)$ from Lemma \ref{lem:I_factors}, $\mcI(g)$ being non-empty implies that both $\mcI(g^\circ)$ and $\mcI(g^\et)$ are non-empty. Since $\mbY^\circ$ has no étale factor by definition, \cite[Proposition 4.23 (1)]{LM2} applies and states that $z_g^\circ$ is topologically nilpotent. Then $\mcI(g^\circ) = \mcN^\circ \cap \mcZ(z_g^\circ)$ by Lemma \ref{lem:Ig}. Writing $g = \left(\begin{smallmatrix} x & πy \\ y & x\end{smallmatrix} \right)$ as in \eqref{eq:endo_ring_D_explicit} (here we used that we consider the case $λ = 1/2$), we obtain from
\begin{equation}\label{eq:z_g_general}
z_g^\circ =  \begin{pmatrix}
 & πw^\circ \\ w^\circ & 
\end{pmatrix},\quad\ \ w = x^{-1}y,
\end{equation}
that $\mcI(g^\circ) = \mcZ(w^\circ)$ where the right hand side is a closed formal subscheme of $\mcN$. Thus $\mcI(g^\circ) \neq \emptyset$ implies that $w^\circ$ is integral. Since $(z_g^\circ)^2/π = (w^\circ)^2$, this shows that $(z_g^\circ)^2/π$ is integral as claimed.

We are left to prove that $(z_g^\et)^2/π$ is integral. Passing from $\mbY^\et$ and $\mbX^\et$ to Tate modules, we may identify $\mcI(g^\et)(\mbF)$ with the set $\mcL(\varpi^{-1} z_g^\et)$ from Definition \ref{L_b}. By Lemma \ref{vanish_b}, this set being non-empty implies $(z_g^\et)^2/π$ integral. The proof is now complete.
\end{proof}

\begin{lem}\label{lem:top_nilp}
Let $g^\circ\in G^\circ_{1/2, \mbY, \rs}$ be a regular semi-simple element such that $(z^\circ_g)^2/π$ is integral. Then $(z_g^\circ)^2/π$ is even topologically nilpotent.
\end{lem}
\begin{proof}
We consider the two cases $n^\circ$ even or odd separately. Assume first that $n^\circ$ is even. Then $D_{1/2, \mbY}^\circ$ is a CDA over $F$ of degree $2n^\circ$ and Hasse invariant $(n^\circ + 1)/2n^\circ$. Thus $L_g^\circ$ is a field extension of degree $n^\circ$ of $F$. As $E\tensor_F L_g^\circ$ embeds into $D_{1/2, \mbY}^\circ$, this tensor product has to be a field, so the inertia degree of $L_g$ over $F$ is odd. This implies that its ramification index is even. Moreover, $B_g^\circ$ embeds into $D_{1/2, \mbY}^\circ$ and is hence a division algebra. It follows that $(z_g^\circ)^2\in (L_g^\circ)^\times$ is not a norm from $E\tensor_F L_g^\circ$ and hence has odd valuation. Since the ramification index is even, $(z_g^\circ)^2/π$ has odd valuation as well. So $(z_g^\circ)^2/π\in O_{L_g^\circ}^\times$ is impossible as was to be shown.

Now we consider the case that $n^\circ$ is odd. Then $D_{1/2, \mbY} \iso M_2(Q)$ where $Q$ is a CDA over $F$ of degree $n^\circ$ and Hasse invariant $(n^\circ + 1)/2n^\circ$. The étale $F$-algebra $L_g^\circ$, which has degree $n^\circ$, again has to be a field because there cannot exist an embedding $L_g\to M_2(Q)$ otherwise. (This argument used that $n^\circ$ is odd.) Then one obtains that $B_g^\circ \iso M_2(L_g^\circ)$ because it equals the centralizer of $L_g^\circ$ in $M_2(Q)$. This means that $(z_g^\circ)^2\in L_g^\circ$ is a norm from $E\tensor_FL_g^\circ$ which is equivalent to $(z_g^\circ)^2$ having even valuation. The degree $[L_g^\circ:F]$ is odd, so it follows that $(z_g^\circ)^2/π$ has odd valuation and thus cannot lie in $O_{L_g^\circ}^\times$ as was to be shown.
\end{proof}

\begin{construction}\label{constr:factors}
Assume that $g = (g^\circ, g^\et) \in G_{1/2, \mbY, \mr{rs}}$ satisfies that $z_g^2/π$ is integral. By Lemma \ref{lem:top_nilp}, $(z_g^\circ)^2/π$ is even topologically nilpotent. Let $R = O_F[z_g^2/π] \subset C_\mbY$ be the $O_F$-algebra generated by $z_g^2/π$. Using Hensel's Lemma as in Lemma \ref{lemm: hensel}, there is a unique factorization $R = R_0 \times R_1$ such that, when writing $z_g^2/π = (ζ_0, ζ_1)$, the component $ζ_0$ is topologically nilpotent and $ζ_1\in R_1^\times$. Since $(z^\circ_g)^2/π$ is topologically nilpotent, the projection $R\to R_1$ factors through the projection map
$$F[z_g^2/π] = F[(z_g^\circ)^2/π]\times F[(z_g^\et)^2/π] \lr F[(z_g^\et)^2/π].$$
Let $\mbY = \mbY_0\times \mbY_1$ be the decomposition of $\mbY$ up to isogeny with respect to the idempotents defining $R = R_0\times R_1$. By what was just said, $\mbY_1$ is an étale $π$-divisible $O_E$-module. The centralizer in $G_{1/2, \mbY, \mr{rs}}$ of the idempotent $(1,0)\in R$ is thus of the form $J_0 \times J_1$ with
$$J_0 = G_{1/2, \mbY_0}\quand J_1 = G_{1/2, \mbY_1} \iso GL_{n_1}(B).$$
Here, $n_1$ is the $O_E$-height of $\mbY_1$. Let $(g_0, g_1)\in J_{0, \mr{rs}} \times J_{1, \mr{rs}}$ be a pair of regular semi-simple elements such that $H_\mbY (g_0, g_1)H_\mbY = H_\mbY gH_\mbY$. Such a pair exists: For example, after an $H_\mbY$-translation, we may assume that $g = 1 + z_g$ in which case $g$ commutes with $z_g$. Then $g$ itself lies in $J_0 \times J_1$.
\end{construction}

\begin{prop}\label{prop:factorization_geometric}
Assume $g\in G_{1/2, \mbY, \mr{rs}}$ is such that $z_g^2/π$ is integral over $O_F$. Let $(g_0, g_1) \in J_{0, \mr{rs}}\times J_{1, \mr{rs}}$ be as in Construction \ref{constr:factors}. Then the intersection number of $g$ factors as
\begin{equation}\label{eq:factorization_geometric}
\mr{Int}(g) = \mr{Int}(g_0) \Orb(g_1, 1_{GL_{n_1}(O_B)}).
\end{equation}
\end{prop}
\begin{proof}
All three quantities in \eqref{eq:factorization_geometric} only depend on the invariants of the three elements $g$, $g_0$ and $g_1$. So we may assume that, $g$ has the form $\left(\begin{smallmatrix} 1 & πy \\ y & 1\end{smallmatrix}\right)$. By Lemma \ref{lem:Ig},
$$\mcI(g) = \mcN\cap \mcZ\left(\left(\begin{smallmatrix}  & πy \\ y & \end{smallmatrix}\right)\right) = \mcZ(y)$$
where $\mcZ(y)\subset \mcN$ is the subspace of all $(Y, ι, ρ)$ such that $ρyρ^{-1}\in \End(Y)$. Moreover, since $y^2 = z_g^2/π$, the ring $R$ in Construction \ref{constr:factors} agrees with $O_F[y^2]$. Using its idempotents, every $(Y, ι, ρ)\in \mcZ(y)$ factors as $(Y_0, ι_0, ρ_0) \times (Y_1, ι_1, ρ_1)$ where the two triples lie in $\mcI(g_0)$ and $\mcI(g_1)$. Furthermore, we may choose $Γ\subset L^\times_g = F[πy^2]^\times$ as $Γ = Γ_0\times Γ_1$ with $Γ_i \subset F[z_{g_i}^2]^\times$. In this way, we obtain the factorization
\begin{equation}\label{eq:fact_abstract}
Γ\backslash \mcI(g) \iso (Γ_0 \backslash \mcI(g_0)) \times_{\Spf O_{\breve F}} (Γ_1 \backslash \mcI(g_1)).
\end{equation}
Since $\mbY_1$ is an étale $π$-divisible $O_E$-module, $\mcI(g_1)$ is étale over $\Spf O_{\breve F}$. Moreover, passing to Tate modules, $\mcI(g_1)(\mbF)$ may be identified with the set $\mcL(\varpi^{-1} z_{g_1})$ from Definition \ref{L_b}. The formal scheme $Γ_1\backslash \mcI(g_1)$ is then a disjoint union of $\Orb(g_1, 1_{GL_{n_1}(O_B)})$ many copies of $\Spf O_{\breve F}$. Thus \eqref{eq:factorization_geometric} follows from \eqref{eq:fact_abstract}.
\end{proof}

\subsection{Proof of Theorem \ref{thm:main}}
\label{ss:proof_AFL}
\begin{proof}[Proof of Theorem \ref{thm:main} (Vanishing Part)]
Let $γ\in G'_{\mr{rs}}$ be regular semi-simple with invariant $δ = \Inv(γ)$. We claim that there exists a strict $O_E$-module $\mbY$ over $\mbF$ and a matching element $g\in G_{1/2, \mbY, \mr{rs}}$ if and only if $\ord_{1/2}(δ) = 1$. Indeed, this condition is by definition equivalent to
\begin{equation}\label{eq:B_ord_1}
L_δ \iso L_0 \times L_1,\quad B_δ \iso (B_{1/2}\tensor_FL_0) \times B_1
\end{equation}
where $L_0$ is an étale $F$-algebra, $L_1$ a field extension of $F$, and $B_1/L_1$ the quaternion algebra that is not isomorphic to $B_{1/2}\tensor_FL_1$. Assuming this condition is met, let $\mbY$ be the unique strict $O_E$-module such that $n^\circ = [L_1:F]$ and $n^\et = [L_0:F]$. Then
\begin{equation}\label{eq:D_Y}
D_{1/2, \mbY} \iso M_{[L_0:F]}(B) \times D^\circ_{1/2, \mbY}.
\end{equation}
Here, $D^\circ := D_{1/2, \mbY}^\circ$ is a CSA of degree $2n^\circ$ with Hasse invariant $(n^\circ + 1)/2n^\circ$. There are two cases: If $n^\circ$ is even, then $B_1$ and $D^\circ$ are both division algebras. If $n^\circ$ is odd, then $B_1 = M_2(L_1)$ and $D^\circ$ is the ring of $(2\times 2)$-matrices over a CDA. In both cases, there exists an $F$-algebra embedding $B_1\to D^\circ$ and hence an $F$-algebra embedding $B_δ\to D_{1/2, \mbY}$. By \cite[Corollary 2.8]{LM2} (2), resp. its extension to semi-simple $F$-algebras, this is equivalent to the existence of an element $g\in G_{1/2, \mbY, \mr{rs}}$ with $\Inv(g;T) = δ(T)$.

Conversely, assume that there exists an embedding $β:B_δ \to D_{1/2, \mbY}$ where $D_{1/2, \mbY}$ is as in \eqref{eq:D_Y}. Then $β(B_δ)$ agrees with the centralizer of $β(L_δ)$ for dimension reasons \cite[Proposition 2.6 (3)]{LM2} which implies that $B_δ$ takes the form in \eqref{eq:B_ord_1}, and hence that $\mr{ord}_{1/2}(δ) = 1$. This finishes the prove of our claim.

The vanishing statement is now obtained as follows. Assume there is no $\mbY$ with matching $g\in G_{1/2, \mbY, \rs}$. By the claim, this means $\mr{ord}_{1/2}(δ) = 0$ or $\mr{ord}_{1/2}(δ) \geq 2$. In the first case, we apply the functional equation
$$\Orb(γ, f'_{1/2}, -s) = ε_{1/2}(γ) \Orb(γ, f'_{1/2}, s)$$
from \eqref{eq:fctl_equation}. The sign $ε_{1/2}(γ)$ can be seen to equal $(-1)^{\mr{ord}_{1/2}(δ)}$. Thus $\del (γ, f'_{1/2}) = 0$ if $\mr{ord}_{1/2}(δ) = 0$. In the second case, the vanishing of $\del(γ, f'_{1/2})$ is implied by Corollary \ref{cor:vanishing}. 
\end{proof}

\begin{proof}[Proof of Theorem \ref{thm:main} (Reduction to the linear AFL)] Fix $\mbY$ and a regular semi-simple element $g\in G_{1/2, \mbY, \mr{rs}}$. Let $γ\in G'_{\mr{rs}}$ be a matching element. Assuming the linear AFL for all strict $O_E$-modules with the same connected height $n^\circ$, we need to see that
\begin{equation}\label{eq:AFL_to_show}
\del (γ, f'_{1/2}) = 2\,\mr{Int}(g)\log(q).
\end{equation}
The condition that $γ$ and $g$ match is by definition equivalent to assuming that $z_γ^2$ and $z_g^2$ have the same characteristic polynomial. Thus $z_γ^2/π$ is integral if and only if $z_g^2/π$ is integral. By Corollary \ref{vanish_par} and by Lemma \ref{lem:integrality_y}, both sides of \eqref{eq:AFL_to_show} vanish if these elements are not integral. So from now on we assume that $z_g^2/π$ and $z_γ^2/π$ are integral.

We may assume that $γ$ and $g$ take the form $γ = 1 + z_γ$ and $g = 1 + z_g$. Let $γ = (γ_0, γ_1)$ and $g = (g_0, g_1)$ be the components of $γ$ and $g$ such that $z^2_{γ_0}/π$ and $z^2_{g_0}/π$ are topologically nilpotent, and $z_{γ_1}^2/π$ and $z_{g_1}^2/π$ integrally invertible. Then $γ_0$ matches $g_0$ and $γ_1$ matches $g_1$.

We write $f'_{1/2, n_0}$ and $f'_{1/2, n_1}$ for the Parahoric test functions on $GL_{2n_0}(F)$ and $GL_{2n_1}(F)$. By Lemma \ref{aux1lattice}, $\del (γ_1, f'_{1/2,n_1}) = 0$. By Corollary \ref{reduction}, we hence obtain that
\begin{equation}\label{eq:fact_gamma}
\del (γ, f'_{1/2}) = \del(γ_0, f'_{1/2, n_0}) \Orb(γ_1, f'_{1/2, n_1}).
\end{equation}
By Proposition \ref{prop:factorization_geometric}, we also have the factorization
\begin{equation}\label{eq:fact_g}
\mr{Int}(g) = \mr{Int}(g_0)\Orb(g_1, 1_{GL_{n_1}(O_B)}).
\end{equation}
By the fundamental lemma (Theorem \ref{thm:FL}),
$$\Orb(γ_1, f'_{1/2, n_1}) = \Orb(g_1, 1_{GL_{n_1}(O_B)}).$$
Thus, Identity \eqref{eq:AFL_to_show} follows if we can prove
$$\del(γ_0, f'_{1/2, n_0}) = 2\,\mr{Int}(g_0)\log(q).$$
Since $z_{γ_0}^2/π$ and $z_{g_0}^2/π$ are topologically nilpotent, by Theorem \ref{thm:auxiliary}, there are two elements $\wt {γ}_0 \in GL_{2n_0}(F)_{\mr{rs}}$ and $\wt g_0 \in G_{0, \mbY_0, \mr{rs}}$ such that
$$\del(\wt {γ}_0, 1_{GL_{2n_0}(O_F)}) = \del(γ_0, f'_{1/2, n_0}) \quand \mr{Int}(\wt g_0) = \mr{Int}(g_0).$$
The linear AFL for $\mbY_0$, whose connected part has height $n^\circ$, precisely states that
$$\del(\wt {γ}_0, 1_{GL_{2n_0}(O_F)}) = 2\,\mr{Int}(\wt g_0)\log(q)$$
and the proof is complete.
\end{proof}


\begin{thebibliography}{99}
%\bibitem{AC}
%J. Arthur, L. Clozel, \emph{Simple Algebras, Base Change, and the Advanced Theory of the Trace Formula}, Princeton University Press \textbf{AM-120} (1989).

\bibitem{Drinfeld}
V. G. Drinfeld, \emph{Coverings of $p$-adic symmetric domains}, Funkcional. Anal. i Priložen. \textbf{10} (1976), no. 2, 29--40.

%\bibitem{Gross}
%Gross, B.H., 1998. \emph{On the Satake isomorphism}. London Mathematical Society Lecture Note Series, pp.223-238.

%\bibitem{CCO}
%C.-L. Chai, B. Conrad, F. Oort, \emph{Complex multiplication and lifting problems},
%Mathematical Surveys and Monographs \textbf{195}, American Mathematical Society, Providence, RI, (2014).
%
%\bibitem{Clozel}
%L. Clozel, \emph{The fundamental lemma for stable base change}, Duke Math. J. \textbf{61} (1990), no. 1, 255--302.
%
%\bibitem{GS}
%H. Gillet, C. Soulé, \emph{K-théorie et nullité des multiplicités d'intersection}, C. R. Acad. Sci. Paris Sér. I Math. \textbf{300} (1985), no. 3, 71--74.
 
\bibitem{Guo}
J. Guo, \emph{On a generalization of a result of Waldspurger}, Canad. J. Math. \textbf{48} (1996), no. 1, 105--142.
%
%\bibitem{H}
%B. Howard, \emph{Complex multiplication cycles and Kudla--Rapoport divisors}, Ann. of Math. (2) \textbf{176} (2012), no. 2, 1097--1171.

%\bibitem{HS}
%U. Hartl, R. K. Singh, \emph{Local shtukas and divisible local Anderson modules}, Canad. J. Math. \textbf{71} (2019), no. 5, 1163--1207. 

\bibitem{HL}
B. Howard, Q. Li, \emph{Intersections in Lubin--Tate space and biquadratic Fundamental Lemmas}, Preprint (2020), arXiv:2010.07365.

\bibitem{JR}
H. Jacquet, S. Rallis, \emph{Uniqueness of linear periods}, Compositio Math. \textbf{102} (1996), no. 1, 65--123. 

\bibitem{K}
R. Kottwitz, \emph{Base change for unit elements of Hecke algebras}, Compositio Math. \textbf{60} (1986), no.2, 237–250.

%\bibitem{KRZ}
%S. Kudla, M. Rapoport, Th. Zink, \emph{On the $p$-adic uniformization of Shimura curves}, Preprint (2020).
%
%\bibitem{Labesse}
%J.-P. Labesse, \emph{Fonctions élémentaires et lemme fondamental pour le changement de base stable}, Duke Math. J. \textbf{61} (1990), no. 2, 519--530. 
%
%\bibitem{Lau}
%E. Lau, \emph{Displays and formal $p$-divisible groups}, Invent. Math. \textbf{171} (2008), no. 3, 617--628.

\bibitem{LXZ}
S. Leslie, J. Xiao, W. Zhang, \emph{Periods and heights for unitary symmetric pairs}, in preparation.


\bibitem{Li}
Q. Li, \emph{An intersection number formula for CM cycles in Lubin--Tate towers}, Preprint (2018), arXiv:1803.07553, to appear in Duke Math. J.

\bibitem{Li_GL4}
Q. Li, \emph{A computational proof of the linear Arithmetic Fundamental Lemma for $GL_4$}, Preprint (2019), arXiv:1907.00090, to appear in Canad. J. Math.

\bibitem{Li_future}
Q. Li, \emph{A computational proof for the biquadratic Linear AFL for $GL_4$}, in preparation.

\bibitem{LM1}
Q. Li, A. Mihatsch, \emph{On the linear AFL: The non-basic case},
Preprint (2022), arXiv:2208.10144.

\bibitem{LM2}
Q. Li, A. Mihatsch, \emph{Arithmetic Transfer for inner forms of $GL_{2n}$}, Preprint (2023), arXiv:2307.11716.

%\bibitem{Liu}
%Y. Liu, \emph{Fourier--Jacobi cycles and arithmetic relative trace formula (with an appendix by Chao Li and Yihang Zhu)}, Camb. J. Math. \textbf{9} (2021), no. 1, 1--147. 
%
%\bibitem{Messing}
%W. Messing, \emph{The crystals associated to Barsotti-Tate groups: With applications to abelian schemes}, Lecture Notes in Mathematics \textbf{264}, Springer-Verlag, Berlin-New York, 1972.
%%
%\bibitem{Meusers}
%V. Meusers, \emph{Lubin--Tate formal groups}, Astérisque \textbf{312} (2007), 49--55.

\bibitem{M_loc_const}
A. Mihatsch, \emph{Local Constancy of Intersection Numbers}, Algebra Number Theory, \textbf{16} (2022), no. 2, 505--519.

\bibitem{MZ}
A. Mihatsch, W. Zhang, \emph{On the Arithmetic Fundamental Lemma Conjecture over a general $p$-adic field}, Preprint (2021).

%\bibitem{RSZ1}
%M. Rapoport, B. Smithling, W. Zhang, \emph{Arithmetic diagonal cycles on unitary Shimura varieties}, Compos. Math. \textbf{156} (2020), no. 9, 1745--1824.
%
%\bibitem{RSZ2}
%M. Rapoport, B. Smithling, W. Zhang, \emph{Regular formal moduli spaces and arithmetic transfer conjectures}, Math. Ann. \textbf{370} (2018), no. 3-4, 1079--1175.
%
%\bibitem{RSZ3}
%M. Rapoport, B. Smithling, W. Zhang, \emph{On the arithmetic transfer conjecture for exotic smooth formal moduli spaces}, Duke Math. J. \textbf{166} (2017), no. 12, 2183--2336.
%  
%\bibitem{RSZ_unitary}
%M. Rapoport, B. Smithling, W. Zhang, \emph{On Shimura varieties for unitary groups}, Pure Appl. Math. Q. \textbf{17} (2021), no. 2, 773--837.

%\bibitem{RZ1}
%M. Rapoport, Th. Zink, \emph{Period Spaces for $p$-divisible Groups}, Ann. of Math. Stud. \textbf{141}, Princeton University Press, Princeton, NJ, 1996.
%
%\bibitem{RZ2}
%M. Rapoport, Th. Zink, \emph{On the Drinfeld moduli problem of $p$-divisible groups}, Camb. J. Math. \textbf{5} (2017), no. 2, 229--279.
%
%\bibitem{Roberts}
%P. Roberts, \emph{The vanishing of intersection multiplicities of perfect complexes}, Bull. Amer. Math. Soc. \textbf{13} (1985), no. 2, 127--130. 
%
%\bibitem{Stacks}
%The Stacks Project Authors, \emph{Stacks Project} (2018), \href{https://stacks.math.columbia.edu}{https://stacks.math.columbia.edu}.
%
%\bibitem{V}
%Rapoport M, Viehmann E. Towards a theory of local Shimura varieties[J]. arXiv preprint arXiv:1401.2849, 2014.


\bibitem{Z}
W. Zhang, \emph{On arithmetic fundamental lemmas}, Invent. Math. \textbf{188} (2012), no. 1, 197--252.
%
%\bibitem{Z_survey_12}
%W. Zhang, \emph{Gross--Zagier formula and arithmetic fundamental lemma}, Fifth International Congress of Chinese Mathematicians, pt. 1, 2, 447--459, Amer. Math. Soc., Providence, RI, 2012. 

\bibitem{Z_survey_18}
W. Zhang, \emph{Periods, cycles, and L-functions: a relative trace formula approach}, Proceedings of the International Congress of Mathematicians---Rio de Janeiro 2018. Vol. II. Invited lectures, 487--521, World Sci. Publ., Hackensack, NJ, 2018.
 
\bibitem{Z19}
W. Zhang, \emph{Weil representation and arithmetic fundamental lemma}, Ann. of Math. (2) \textbf{193} (2021), no. 3, 863--978.
%
%\bibitem{Z_Bessel}
%W. Zhang, \emph{More Arithmetic Fundamental Lemma conjectures: the case of Bessel subgroups}, Preprint (2021), arXiv:2108.02086.

\bibitem{ZZ}
Z. Zhang, \emph{Maximal parahoric arithmetic transfers, resolutions and modularity} Preprint (2021), arXiv:2112.11994.

%\bibitem{Z}
%T. Zink, \emph{The display of a formal $p$-divisible group}, in \emph{Cohomologies $p$-adiques et applications arithmétiques, I}, Astérisque \textbf{278} (2002), 127--248.

\end{thebibliography}
\end{document}